\documentclass{article}

\usepackage{epsfig}
\usepackage{latexsym,amsmath,amsfonts,amscd}
\usepackage{amssymb,multirow}
\usepackage{graphics}
\usepackage{epsfig}
\usepackage{subfigure}
\usepackage{verbatim}
\usepackage{epstopdf}
\usepackage{color,leftidx}
\usepackage{amsthm}
\usepackage{mathrsfs}

\usepackage{indentfirst}
\usepackage{textcomp}

\usepackage{graphicx}
\usepackage{booktabs, longtable}
\usepackage{titletoc}
\usepackage{graphics}
\usepackage{tabularx}
\usepackage[subfigure,caption2]{ccaption}
\usepackage{enumerate}
\usepackage{cite}
\usepackage{stmaryrd}  
\usepackage{algorithm}
\usepackage{algpseudocode}

\newcommand{\llkh}{ \{\!\!\{ }
\newcommand{\rrkh}{ \}\!\!\} }

\newtheorem{remark}{Remark}[section]
\newtheorem{theorem}{Theorem}[section]
\newtheorem{lemma}{Lemma}[section]
\newtheorem{example}{Example}[section]

\numberwithin{equation}{section}

\usepackage{fancyhdr}

\def\d{\displaystyle}
\def\p{\partial}

\usepackage{authblk}

\begin{document}

\title{A fully discrete LBRFD–IPDG method for linear fourth-order parabolic equations}


\author[1]{Hongying Huang}
\author[2]{Hong Zhang}
\author[3,*]{Yanming Zhang}


\affil[1]{School of Arts and Sciences, Guangzhou Maritime University, Guangzhou, 510725, Guangdong, China. email: huanghy@lsec.cc.ac.cn}%
\affil[2]{Key Laboratory of NSLSCS, Ministry of Education, Jiangsu International Joint Laboratory of BDMCA, School of Mathematical Sciences, Nanjing Normal University, 210023, Nanjing, China. Email: hzcomath@163.com}
\affil[3]{School of Mathematics and Statistics, Hunan
First Normal University, Changsha, 410205, Hunan, China.}

\affil[*]{Corresponding author: zhang\_yanming23@163.com }

\maketitle

\begin{abstract}
We propose a fully discrete method for linear fourth-order parabolic equations with Dirichlet boundary conditions, { combining an implicit linear barycentric rational finite difference (LBRFD) multistep scheme in time with a mixed interior penalty discontinuous Galerkin method in space. For certain parameter pairs $(n,d)$, the LBRFD method is shown to be $A(\alpha)$-stable with a wider stability angle than the corresponding BDF$p$ method of the same order. Stability and a priori error estimates are established via a $G$-energy technique, the discrete Gr\"onwall lemma. In one dimension, the analysis yields an optimal $L^2$ estimate of order $h^{k+1}+\tau^p$, with $p=d$ for even $n-d$ and $p=d+1$ otherwise, owing to the design of suitable penalty terms and the construction of the projection $\Pi_h$. In two dimensions, the temporal order remains optimal while the spatial order is theoretically two orders lower, since no projection as favorable as the one-dimensional $\Pi_h$ is available on triangular meshes. Despite this reduction, numerical experiments confirm optimal convergence of order $h^{k+1}+\tau^p$, and further show that the penalty term is indispensable, while the startup procedure leaves the convergence order unchanged.}  
\end{abstract}

%

{\bf Keywords:}{ Barycentric rational finite difference; Fourth-order parabolic equation; Discontinuous Galerkin method; Error estimate }

\section{Introduction}\label{sec_1}

In this paper, we study the linear barycentric rational finite difference (LBRFD) method and the mixed discontinuous Galerkin (DG) method for the linear time-dependent fourth-order equation
\begin{equation}\label{eq:original_problem1}
u_t+\Delta^2 u=f(\mathbf{x},t),\quad (\mathbf{x},t)\in\Omega\times (0,T], \quad u(\mathbf{x},0)=u_0(\mathbf{x}),\quad \mathbf{x}\in\Omega,
\end{equation}
where $\Omega\subset\mathbb{R}^d$ is a bounded polygonal domain. Fourth-order models of this type arise in a wide range of applications, including thin beams and plates, strain-gradient elasticity, and phase separation in binary mixtures \cite{han1999dynamics}.

In recent years, the discontinuous Galerkin (DG) method has attracted considerable attention for its ability to treat nonsmooth solutions and its inherent adaptivity \cite{shu2014discontinuous}. For steady-state fourth-order boundary-value problems, Baker \cite{baker1977finite} pioneered the use of DG to approximate the biharmonic equation with homogeneous Dirichlet boundary conditions. Since then, a variety of DG formulations have been developed, including the popular $C^0$ interior penalty DG (IPDG) method \cite{brenner2012quadratic,engel2002continuous}, IPDG \cite{dong2021residual,georgoulis2009discontinuous,mozolevski2003priori,suli2007hp,huang2026penalty}, direct DG methods \cite{huang2026symmetric}, mixed DG (MDG) methods \cite{gudi2008mixed,liu2018mixed,liu2023sav}, and the single face-hybridizable DG method \cite{cockburn2009hybridizable}, among others.

For time-dependent fourth-order problems, however, the literature is comparatively sparse, especially for non-homogeneous boundary-value problems. Several DG methods have been proposed for Cahn--Hilliard-type equations \cite{feng2007fully,medina2022stabilized}. Dong and Shu \cite{dong2009analysis} applied the local discontinuous Galerkin (LDG) method to \eqref{eq:original_problem1} with periodic boundary conditions and derived optimal error estimates on Cartesian and triangular meshes. In \cite{meng2012superconvergence}, a minimal-dissipation LDG scheme with suitable boundary penalty terms was numerically investigated for the one-dimensional version of \eqref{eq:original_problem1} with Dirichlet boundary conditions, and optimal convergence rates were observed. Liu and Yin \cite{liu2018mixed} developed a penalty-free mixed DG scheme for \eqref{eq:original_problem1} and proved stability together with optimal error estimates under periodic boundary conditions; for Dirichlet and Navier boundary conditions, penalty terms are needed for stability, and optimal error analysis is still outstanding. In \cite{liu2020analysis,tao2020ultraweak}, ultraweak LDG methods were applied to fourth-order partial differential equations and were shown to exhibit superconvergence. Subsequent developments of the ultraweak LDG method further simplified the design of numerical fluxes for fourth-order problems with four types of boundary conditions \cite{fu2025analysis}. {  For two-dimensional fourth-order problems with Dirichlet boundary conditions, \cite{fu2025analysis} derived optimal error estimates for the spatial discretization by designing suitable boundary penalty terms and constructing the projection $\Pi_h$ as a tensor product built from the one-dimensional projection $\Pi_h$ within the $Q^k$ framework. However, no such result is available on triangular meshes.}

Regarding temporal discretization, classical approaches include explicit Runge--Kutta methods \cite{liu2020analysis}, implicit Runge--Kutta methods \cite{liu2022high,fu2025analysis}, the backward Euler method \cite{liu2018mixed}, the Crank--Nicolson method \cite{kadri2018fourth,doss2012galerkin},  BDF$p$ multistep schemes  \cite{danumjaya2006numerical,chen2019second,gudi2013fully} and spectral deferred correction \cite{tao2020ultraweak}.   In the context of linear fourth-order principal parts, the implicit treatment of diffusion offers a notable advantage by greatly relaxing the time-step restriction inherent in explicit schemes. When an equation contains both stiff diffusion and non-stiff convection or reaction, implicit--explicit (IMEX) Runge--Kutta or multistep schemes have become a mainstream choice: lower-order terms are treated explicitly and higher-order diffusion implicitly, thereby preserving high-order temporal accuracy while avoiding parabolic CFL restrictions \cite{ascher1995implicit,pareschi2005implicit}. Wang, Shu, and Zhang systematically analyzed the coupled stability of LDG spatial discretization with IMEX time marching \cite{wang2016local}. 
For one-dimensional linear fourth-order equations, IMEX--LDG schemes with implicit fourth-order terms and explicit lower-order terms were proved to be unconditionally energy stable and optimally accurate \cite{wang2017stability}.

All of the temporal discretization schemes discussed above rely on classical polynomial interpolation, which performs well only for small polynomial degrees $k$. In fact, customary polynomial interpolants are ill-conditioned and may suffer from Runge's phenomenon when the interpolation nodes are equispaced \cite{dahlquist1956convergence}. To overcome this drawback, we turn to linear barycentric rational interpolants (LBRIs), first introduced by Berrut~\cite{berrut1988rational} and later generalized by Floater and Hormann~\cite{floater2007barycentric}. These interpolants are linear in the data, infinitely smooth, and free of real poles. Moreover, the generalized construction by Floater and Hormann achieves arbitrarily high convergence orders for most sets of interpolation points, including equispaced ones, which has dramatically increased the attention they have received. Owing to their high convergence order and excellent numerical stability, LBRIs offer a more practical alternative for constructing numerical schemes. Abdi et al.~\cite{abdi2019adaptive} proposed and analyzed two simple multistep methods for stiff ordinary differential equations, employing a starting procedure based on linear rational interpolation~\cite{floater2007barycentric} and linear barycentric rational finite difference (LBRFD) formulas~\cite{klein2012linear}. They also designed an adaptive version with one free parameter to enhance stability and demonstrated that, for certain choices of $(n,d)$, the LBRFD scheme is $A(\alpha)$-stable with a wider stability angle than the corresponding BDF$p$ method.
 In \cite{abdi2021second}, Abdi and Hojjati introduced a family of linear multistep second-derivative methods together with a starting procedure based on barycentric rational interpolants for ODEs, and established their order of convergence and linear stability properties. Esmaeelzadeh et al.\ \cite{esmaeelzadeh2021ebdf} designed BDF-type methods based on these interpolants as a general class with improved accuracy and stability. More recently, Abdi et al.\ \cite{abdi2025linear} analyzed a family of BDF-type methods based on LBRIs on nonuniform grids and developed a variable-stepsize/variable-order algorithm for ODEs. LBRIs have also been used to construct numerical methods for Volterra integral and integro-differential equations \cite{abdi2022explicit,abdi2018barycentric,abdi2020numerical}.

 However, to the best of our knowledge, the LBRFD method, together with a starting procedure based on LBRIs, has not yet been applied to partial differential equations. To fill this gap, we develop a fully discrete LBRFD--IPDG scheme for \eqref{eq:original_problem1} with Dirichlet boundary conditions. { For the one-dimensional problem, we establish stability and prove optimal error estimates in both space and time via a $G$-energy technique, the discrete Gr\"onwall lemma, and a key projection $\Pi_h$. The analysis is then extended to two dimensions on shape-regular triangular meshes with piecewise polynomials of total degree at most $k$. In this setting, the temporal order remains optimal, whereas the spatial order is two orders lower than optimal, since no projection analogous to the one-dimensional $\Pi_h$ is available on triangular meshes; this loss is reflected in the boundary contributions. Despite this theoretical reduction, numerical results confirm that optimal convergence is attained in practice. They further show that the penalty term is indispensable for stability and optimal convergence, and illustrate the influence of the LBRI-based starting procedure as well as the choice of the parameters $(n,d)$.}

The remainder of this paper is organized as follows. Section~2 reviews linear barycentric rational interpolation and the implicit LBRIFD multistep method with its starting procedure for ODEs. Section~3 develops the mixed IPDG spatial discretization and the fully discrete scheme for the one-dimensional problem, and establishes stability and a priori error estimates by a G-energy method and a discrete Gr\"onwall argument. Section~4 extends the scheme and the theoretical results to two dimensions, presenting  the main results. Section~5 presents numerical experiments in one and two dimensions, and Section~6 provides concluding remarks.

Throughout this paper, $\|\cdot\|_E$ and $(\cdot,\cdot)_E$ denote the $L^2(E)$ norm and inner product on a domain $E$, and $\|\cdot\|_{s,E}$ and $|\cdot|_{s,E}$ denote the $H^s(E)$ norm and seminorm, respectively. For notational convenience, when $E=\Omega$ we drop the subscript $\Omega$ and write $(\cdot,\cdot)$, $\|\cdot\|$, $\|\cdot\|_{s}$, and $|\cdot|_{s}$
for $(\cdot,\cdot)_{\Omega}$, $\|\cdot\|_{\Omega}$, $\|\cdot\|_{s,\Omega}$, and $|\cdot|_{s,\Omega}$, respectively.

\section{Multistep method based on LBRIs}

\subsection{A review of linear barycentric rational interpolation}

Linear barycentric rational interpolants (LBRIs) were first introduced by Berrut \cite{berrut1988rational} and later generalized by Floater and Hormann \cite{floater2007barycentric}. Owing to their smoothness, accuracy, and ease of implementation, they are among the most efficient infinitely smooth interpolants, especially on equispaced nodes \cite{berrut2017Linear}. In this subsection, we review the Floater--Hormann interpolants and the error estimates needed in the subsequent development of the BRIFD formula.

Given a function $g$ and $n+1$ distinct nodes $t_i$, $i=0,1,\ldots,n$, with data pairs $(t_i,g(t_i))$, we employ the barycentric rational interpolation of Floater and Hormann \cite{floater2007barycentric}. For every fixed nonnegative integer $d\le n$, define the polynomials $p_i(t)$, $i=0,\ldots,n-d$, that interpolate $g$ at the points $t_i,\ldots,t_{i+d}$. The rational interpolant is then constructed as
\begin{equation}\label{eq:bri_g}
r_n[g](t)=\frac{\sum_{i=0}^{n-d}\lambda_i(t)p_i(t)}{\sum_{i=0}^{n-d}\lambda_i(t)},
\end{equation}
where
\begin{equation}\label{eq:bri_lambda}
\lambda_i(t):=\frac{(-1)^i}{(t-t_i)\cdots(t-t_{i+d})}.
\end{equation}
Substituting the Lagrange form of $p_i(t)$ into \eqref{eq:bri_g}, the interpolant can be written in the barycentric form
\begin{equation}\label{eq:bri_g1}
r_n[g](t)=\d\sum_{i=0}^nl_i(t)g(t_i),\quad l_i(t):=\d\frac{\d\frac{w_i}{t-t_i}}{\d\sum_{j=0}^n\frac{w_j}{t-t_j}},
\end{equation}
with barycentric weights
\begin{equation}\label{eq:bri_weight}
w_i:=\sum_{j\in J_i}(-1)^{j}\d\prod_{{s=j}\atop{s\neq i}}^{j+d}\frac{1}{t_i-t_s},\quad J_i=\{j\in \{0, 1,\ldots,n-d\}: i-d\leq j\leq i\}.
\end{equation}
In particular, for equispaced nodes the weights take the explicit form
\begin{equation}\label{eq:bri_weight_equispace}
w_i:=\frac{(-1)^{i-d}}{2^d}\sum_{j\in J_i}\binom{d}{i-j},\quad J_i=\{j\in \{0, 1,\ldots,n-d\}: i-d\leq j\leq i\}.
\end{equation}
As $\tau\to 0$, where $\tau:=\max_{0\leq i\leq n-1}|t_{i+1}-t_i|$ denotes the global mesh size, the interpolants \eqref{eq:bri_g1} converge at the rate $O(\tau^{d+1})$ \cite{floater2007barycentric}.

Let $\|g\|_{\infty}:=\max_{t_0\le t\le t_n}|g(t)|$. The following theorem from \cite{floater2007barycentric} provides an interpolation error bound for \eqref{eq:bri_g1}.
\begin{theorem}\label{thm:interpolation_error}
Suppose $g\in C^{d+2}[t_0,t_n]$ and $d\ge 1$, and let $\tau:=\max_{0\leq i\leq n-1}|t_{i+1}-t_i|$. Then
\begin{equation}
\|r_n[g]-g\|_{\infty}\le \tau^{d+1}\cdot\left\{\begin{array}{ll}
(t_n-t_0)\frac{\|g^{(d+2)}\|_{\infty}}{d+2},\quad & \delta=1,\\[2pt]
(t_n-t_0)\frac{\|g^{(d+2)}\|_{\infty}}{d+2}+\frac{\|g^{(d+1)}\|_{\infty}}{d+1},\quad & \delta=0,\end{array}\right.
\end{equation}
where $\delta=1$ for odd $n-d$ and $\delta=0$ for even $n-d$.
\end{theorem}

To construct finite difference approximations, we next introduce the differentiation matrices associated with the basis functions $l_j(t)$ in \eqref{eq:bri_g1}. Let $d_{ij}^{(k)}$ denote the $k$th derivative of $l_j(t)$ evaluated at the node $t_i$. The differentiation matrix $D^{(k)}=(d_{ij}^{(k)})$ is given by the following recurrence relations \cite{klein2012linear}: for the first-order derivative,
\begin{equation}\label{eq:dij1}
d_{ij}^{(1)}:=\left\{\begin{array}{ll}
\d\frac{w_j}{w_i}\frac{1}{t_i-t_j}, & i\neq j,\\[4pt]
-\d\sum_{s=0,s\neq i}^nd_{is}^{(1)},& i=j,
\end{array}\right.
\end{equation}
and, for higher-order derivatives with $k\geq 2$,
\begin{equation}\label{eq:dijk}
d_{ij}^{(k)}:=\left\{\begin{array}{ll}
\d\frac{k}{t_i-t_j}\left(\frac{w_j}{w_i}d_{ii}^{(k-1)}-d_{ij}^{(k-1)}\right), & i\neq j,\\[4pt]
-\d\sum_{s=0,s\neq i}^nd_{is}^{(k)},& i=j.
\end{array}\right.
\end{equation}

The convergence rates of the $k$th-derivative approximations of the barycentric rational interpolant $r_n[g]$ for a function $g$ at the interpolation nodes, based on \eqref{eq:dij1}--\eqref{eq:dijk}, were established for $k=1$ and $k=2$ in \cite{berrut2011Convergence}; the authors conjectured the rates for general $k$ with $1\le k\le d$, which was subsequently proved for well-spaced interpolation nodes in \cite{Cirillo2017Convergence}.

\begin{theorem}\label{thm:derivative_error}
Suppose $g\in C^{d+1+k}[t_0,t_n]$, $d\ge 1$, and let $\tau:=\max_{0\leq i\leq n-1}|t_{i+1}-t_i|$. Then, for any set of well-spaced interpolation nodes and for $k=1,2,\ldots,d$, it holds
\begin{equation}\label{eq:derivative_approx1}
\left|r_n[g]^{(k)}(t_i)-g^{(k)}(t_i)\right|\leq C_d \tau^{d+1-k},\quad 0\leq i\leq n,
\end{equation}
and
\begin{equation}\label{eq:derivative_approx2}
\left|r_n[g]^{(k)}(t)-g^{(k)}(t)\right|\leq C_d \tau^{d+1-k}(1+\ln n),\quad t\in [t_0,t_n],
\end{equation}
where $C_d$ is a generic constant depending only on $d$, $k$, and the derivatives of $g$ on $[t_0,t_n]$.
\end{theorem}

\subsection{Linear barycentric rational interpolation finite difference method}
\label{sec:LBRIFD}

In \cite{klein2012linear,klein2012applications}, LBRFD formulas derived from the Floater--Hormann family of linear barycentric rational interpolants \eqref{eq:bri_g1} were introduced. They contain classical polynomial finite differences as a special case and are significantly more stable and accurate than polynomial formulas for one-sided derivative approximations near the endpoints of an interval. Following \cite{abdi2019adaptive}, we present an implicit multistep method based on these formulas on equispaced nodes.

To illustrate the idea, we first consider the stiff initial-value problem for a system of ODEs,
\begin{equation}\label{eq:ODEs}
\begin{cases}
y'(t)=g(y),& t\in (t_0,T],\\
y(t_0)=y_0,
\end{cases}
\end{equation}
where $g:\mathbb{R}^D\to\mathbb{R}^D$ and $D$ denotes the dimension of the system. Let
\[
t_0<t_1<\cdots<t_N=T
\]
be a uniform partition of $[t_0,T]$ with constant stepsize $\tau=t_{i+1}-t_i$, $i=0,1,\ldots,N-1$. Our goal is to approximate the solution of \eqref{eq:ODEs} by a multistep method based on LBRFD formulas.

Suppose that the values $y(t_i)$ are known for $i\le m$, where $m$ and $n$ are positive integers with $n\le m$. For equispaced nodes $t_0,t_1,\ldots,t_n$, it follows from \eqref{eq:bri_weight_equispace} and \eqref{eq:dij1} that
\begin{equation}\label{eq:dij_equispace}
d_{ij}:= \tau d_{ij}^{(1)}=\left\{\begin{array}{ll}
\d\frac{(-1)^{j-i}\d\sum_{s\in J_j}\binom{d}{j-s}}{(i-j)\d\sum_{s\in J_i}\binom{d}{i-s}},&  i\neq j,\\[6pt]
-\d\sum_{s=0,s\neq i}^nd_{is}^{(1)},& i=j,
\end{array}\right.
\quad i,j=0,1,\ldots,n.
\end{equation}
Since $d_{ij}$ depends only on the index differences $i-j$, $i-s_1$, and $j-s_2$ with $s_1\in J_i$ and $s_2\in J_j$, the coefficients are invariant under index shifts. In particular, for the equally spaced nodes $t_{m-n}, t_{m-n+1}, \ldots, t_m$, we have $d_{m-n+i,\,m-n+j}=d_{ij}$. Consequently, the LBRFD approximation of the first derivative of $y$ at these nodes can be written as
\begin{equation}\label{eq:dy_approx}
y'(t_{m-n+i})\approx \frac{1}{\tau}\sum_{j=0}^nd_{ij}y(t_{m-n+j}),\quad i=0,1,\ldots,n,
\end{equation}
where the coefficients $d_{ij}$ are given by \eqref{eq:dij_equispace}.

Applying the left one-sided formula with $i=n$ in \eqref{eq:dy_approx} to \eqref{eq:ODEs}, we obtain the LBRFD
\begin{equation}\label{eq:FD_form}
\sum_{j=0}^nd_{nj}y_{m-n+j}=\tau g(y_m),\quad m=n+1,\ldots,N,
\end{equation}
where $y_m$ denotes an approximation to $y(t_m)$. The coefficients $(d_{n0},d_{n1},\ldots,d_{nn})$ in \eqref{eq:FD_form} for several pairs $(n,d)$ are listed in Table~\ref{tab:LBRFDM_coef}.
Taking $m=n$ in \eqref{eq:dy_approx}, so that the nodes are $t_0,t_1,\ldots,t_n$, and applying the resulting formulas with $i=m$ for $m=1,2,\ldots,n$ to \eqref{eq:ODEs}, we obtain the following starting procedure for $y_1,y_2,\ldots,y_n$:
\begin{equation}\label{eq:FD_form_st}
\sum_{j=0}^nd_{mj}y_j=\tau g(y_m),\quad m=1,2,\ldots,n,
\end{equation}
which yields a fully linear rational multistep method. We first solve the linear system \eqref{eq:FD_form_st}, consisting of $nD$ equations in $nD$ unknowns, to obtain $y_1,\ldots,y_n$; then we march the solution forward using \eqref{eq:FD_form} for $m=n+1,\ldots,N$.

By Theorem~\ref{thm:derivative_error}, for $y\in C^{d+2}[t_0,T]$, the LBRFD formula \eqref{eq:dy_approx} approximates $y'$ with order $d$. Moreover, Theorem~\ref{thm:interpolation_error} shows that on equispaced nodes the interpolation error bound contains an extra factor $t_n-t_0=n\tau$ whenever $n-d$ is odd. This additional factor accounts for the increase of the convergence order of the LBRFD scheme \eqref{eq:FD_form}, together with the startup procedure \eqref{eq:FD_form_st}, from $d$ to $d+1$ in that case. Consequently, for the scheme \eqref{eq:FD_form} equipped with the starting procedure \eqref{eq:FD_form_st}, the overall convergence order is $d$ if $n-d$ is even and $d+1$ if $n-d$ is odd.

\begin{table}[!h]
\caption{Coefficients $(d_{n0},d_{n1},\ldots,d_{nn})$ in LBRFD scheme\eqref{eq:FD_form} for $(n,d)$}
\label{tab:LBRFDM_coef}
\centering{
\begin{tabular}{ccc}
\hline
Order&$(n,d)$&$d_{n0},d_{n1},\ldots,d_{nn}$  \\ \hline
1& (3,1)& -1/3, 1,  -2,  4/3    \\
2& (4,1)& 1/4, -2/3,  1, -2,  17/12   \\
3& (5,2)&-1/5, 3/4,  -4/3,  2, -3, 107/60   \\
4& (6,3)&1/6, -4/5,  7/4,  -8/3,  7/2, -4,  41/20    \\
5& (7,4)& -1/7, 5/6, -11/5,  15/4,  -5,   11/2,  -5, {  949/420}\\
5 &(9,4)&-1/9, 5/8, -11/7,  5/2,  -16/5,  4,  -5,  11/2,  -5, {  5689/2520}   \\
6& (8,5)&1/8, -6/7, 8/3, -26/5,  15/2, -26/3,  8,  -6, {  681/280} \\
6 &(10,5)&1/10, -2/3, 2, -26/7,  31/6,  -32/5,  31/4, -26/3,  8,  -6, {  1021/420}  \\
\hline
\end{tabular}}
\end{table}

According to \cite{abdi2019adaptive}, the LBRFDM \eqref{eq:FD_form} is zero-stable for every choice of $n>d$ with $n<20$ and $d<6$, except for $(n,d)=(7,6), (8,6), (11,6)$, and $(12,6)$.
Table~\ref{tab:LBRFDM_BDF} reports the angles $\alpha$ of $A(\alpha)$-stability for several pairs $(n,d)$ and compares them with those of BDF$p$ methods of the same order. For some choices of $(n,d)$, the LBRFDM is $A(\alpha)$-stable with a wider stability angle than the corresponding BDF$p$ method.

\begin{table}[!h]
\caption{Angles $\alpha$ of $A(\alpha)$-stability for BDFs and LBRFDM}
\label{tab:LBRFDM_BDF}
\centering{
\begin{tabular}{cccccccc}
\hline
&& \multicolumn{6}{c}{LBRFDM}\\\cmidrule{3-8}
Order&$\alpha$ for BDF&$(n,d)$&$\alpha$&$(n,d)$&$\alpha$&$(n,d)$&$\alpha$   \\ \hline
1&$90^\circ$&(3,1)&$90^\circ$&(5,1)&$90^\circ$&(7,1)&$90^\circ$\\
2&$90^\circ$&(4,1)&$90^\circ$&(6,1)&$90^\circ$&(4,2)&$90^\circ$\\
3&$86.032367^\circ$&(5,2)&$88.212710^\circ$&(11,2)&$89.575107^\circ$&(5,3)&$72.732477^\circ$\\
4&$73.351670^\circ$&(6,3)&$80.394530^\circ$&(8,3)&$77.263132^\circ$&(6,4)&$46.806570^\circ$\\
5&$51.839756^\circ$&(7,4)&$54.391987^\circ$&(9,4)&$57.044865^\circ$&(13,4)&$60.284561^\circ$\\
6&$17.839778^\circ$&(8,5)&$18.439738^\circ$&(10,5)&$27.101853^\circ$&(12,5)&$22.560396^\circ$\\
\hline
\end{tabular}}
\end{table}

As in \cite{lubich2013backward}, the stability and error analysis of the LBRFD scheme \eqref{eq:FD_form} for PDE time discretization relies on Dahlquist's G-stability theory \cite{dahlquist1978g} and the multiplier technique of Nevanlinna and Odeh \cite{nevanlinna1981multiplier}, originally developed for contractive nonlinear ODEs.

Take $\mu(\zeta):=\zeta^n-\eta\zeta^{n-1}$ and $\rho(\zeta):=d_{n0}+d_{n1}\zeta+\cdots+d_{nn}\zeta^n$. Following the analysis of BDF$p$ methods in \cite{nevanlinna1981multiplier}, we determine the admissible values of $\eta$ for selected pairs $(n,d)$ by the modified root-locus method. Let $\tilde{\rho}(\zeta)=\zeta^{-n}\rho(\zeta)$.  Define
\begin{equation}\label{eq:root_locus}
\left\{
\begin{array}{l}
x(\theta):=\mathrm{Re}\tilde{\rho}(e^{i\theta}),\\
y(\theta):=\sin\theta\ \mathrm{Im}\tilde{\rho}(e^{i\theta})-\cos\theta\ \mathrm{Re}\tilde{\rho}(e^{i\theta}),
\end{array}\right.
\end{equation}
and plot the curve $\{(x(\theta),y(\theta))\}$ for $\theta\in [0,\pi]$, which is called the modified root-locus curve. According to \cite{nevanlinna1981multiplier}, if the line $x=-\eta y$ lies below this modified root-locus and $\eta\in [0,1)$, then condition 
\begin{equation}\label{eq:stable_condition}
\mathrm{Re}\frac{\rho(\zeta)}{\zeta^n-\eta\zeta^{n-1}}>0, \quad |\zeta|>1
\end{equation}
holds.

 For each pair \((n,d)\), we first determine the tangent line from the origin to the modified root-locus curve, and then find the tangent line satisfying the condition that \(x=-\eta y\) lies below this modified root-locus, thereby yielding the value of \(\eta\). Following this approach, we obtain the values of \(\eta\) for several pairs \((n,d)\), which are listed in Table~\ref{tab:LBRFDM_BDF_eta}.

\begin{table}[!h]
\caption{Values of $\eta$ for BDFs and LBRFDM}
\label{tab:LBRFDM_BDF_eta}
\centering{
\begin{tabular}{cccccccc}
\hline
&& \multicolumn{6}{c}{LBRFDM}\\\cmidrule{3-8}
Order&$\eta$ for BDF&$(n,d)$&$\eta$&$(n,d)$&$\eta$&$(n,d)$&$\eta$   \\ \hline
1& 0&(3,1)&0 &(5,1)&0 &(7,1)&0 \\
2& 0&(4,1)&0 &(6,1)&0 &(4,2)&0 \\
3& 0.083592&(5,2)&0.057022 &(11,2)&0.029111 &(5,3)&0.335984 \\
4& 0.287806&(6,3)&0.207659 &(8,3)&0.223084 &(6,4)&1.356062 \\
5& 0.815980&(7,4)&1.381170 &(9,4)&0.722491&(13,4)&0.793995 \\
6& 5.013188&(12,5)&16.004708&(14,5)&27.235098 &(16,5)& 10.147028  \\
\hline
\end{tabular}}
\end{table}

%

Referring to Dahlquist's G-stability theory (Theorem 3.3 in \cite{dahlquist1978g}) and \eqref{eq:stable_condition}, we have the following lemma.

\begin{lemma}\label{lem:g_stable}
Let \(\mu(\zeta):=\zeta^n-\eta\zeta^{n-1}\) and \(\rho(\zeta):=d_{n0}+d_{n1}\zeta+\cdots+d_{nn}\zeta^n\), and let \((\cdot,\cdot)\) denote an inner product on \(\mathbb{R}^n\). Suppose that the pair \((n,d)\) satisfies \(\eta\in [0,1)\). Then, there exist a symmetric positive definite matrix \(G=(g_{ij})\in\mathbb{R}^{n\times n}\) and real numbers \(r_0,\ldots,r_n\) such that for any sequence of vectors \((v^0,\ldots,v^n) \in \mathbb{R}^{n+1}\),
\[
\left(\sum_{i=0}^n d_{ni}v^i,\, v^n-\eta v^{n-1}\right)=\sum_{i,j=1}^n g_{ij}(v^i,v^j)-\sum_{i,j=1}^n g_{ij}(v^{i-1},v^{j-1})+\left(\sum_{i=0}^n r_iv^i,\,\sum_{i=0}^n r_iv^i\right).
\]
\end{lemma}


\section{One-dimensional Case}
\label{sec:onedimension}

In this section, we propose a fully discrete scheme for the one-dimensional linear fourth-order equation \eqref{eq:original_problem1} by coupling a linear barycentric rational  finite difference (LBRFD) method for temporal discretization with an interior penalty DG  method for spatial discretization.

Consider the following one-dimensional problem :
\begin{equation}\label{eq:original_problem_1D}
\left\{\begin{array}{l}
u_t{ +}u_{xxxx}=f(x,t),\quad x\in (a,b),\ t\in (0,T],\\[2pt]
u(a,t)=f_a(t),\quad u(b,t)=f_b(t),\\[2pt]
u_x(a,t)=g_a(t),\quad u_x(b,t)=g_b(t),\\[2pt]
u(x,0)=u_0(x),\quad x\in [a,b].
\end{array}\right.
\end{equation}

\subsection{Spatial Discretization}

In this section, we discretize the spatial derivatives of \eqref{eq:original_problem_1D} by a mixed IPDG method and establish the $L^2$-stability and error estimate of the resulting semi-discrete scheme. 

By introducing the auxiliary variable $v=-u_{xx}$, the problem \eqref{eq:original_problem_1D} can be rewritten as the following second-order system:
\begin{equation}\label{eq:original_problem_uv}
u_t-v_{xx}=f(x,t),\quad v+u_{xx}=0,\quad x\in \Omega=(a,b).
\end{equation}

Let $a=x_0<x_1<\cdots<x_M=b$ be a partition of $(a,b)$. Denote $I_j=(x_{j-1},x_j)$, $h_j=x_j-x_{j-1}$, and $h=\max_{1\leq j\leq M}h_j$. Assume that there exist positive constants $\kappa_1,\kappa_2$ such that
\[
\kappa_2\leq h_i/h_j\leq \kappa_1,\quad i,j=1,2,\ldots,M.
\]
Let $\mathcal{T}_h=\{I_j\}_{j=1}^M$ be the collection of all subintervals, $\mathcal{N}_h=\{x_j\}_{j=0}^M$ the set of all nodes, $\mathcal{N}_h^{i}=\{x_j\}_{j=1}^{M-1}$ the set of interior nodes, and $\mathcal{N}_h^0=\{a,b\}$ the set of boundary nodes. We define the broken Sobolev space
\[
H^s(\mathcal{T}_h)=\{w\in L^2(a,b): w|_{I_j}\in H^s(I_j),\ \forall I_j\in \mathcal{T}_h\},
\]
and the discontinuous finite element space
\[
V_h^k:=\{w\in L^2(a,b): w|_{I_j}\in P_k(I_j),\ \forall I_j\in \mathcal{T}_h\},
\]
where $P_k(I_j)$ denotes the space of polynomials of degree at most $k$ on $I_j$.

 For $w\in H^s(\mathcal{T}_h)$ with $s\ge 1$, the jump and average at an interior node $x_j\in\mathcal{N}_h^{i}$ are defined by
\begin{align*}
\llbracket w\rrbracket_j =w(x_j^-)-w(x_j^+), \quad \llkh w\rrkh_j=\frac{1}{2}\left(w(x_j^-)+w(x_j^+)\right).
\end{align*}
On the boundary nodes, we set
\[
\llbracket w\rrbracket_0 =-w(a),\quad \llkh w\rrkh_0= w(a),\quad \llbracket w\rrbracket_M =w(b),\quad \llkh w\rrkh_M=w(b).
\]
We also define the discrete energy norm
\begin{equation}\label{eq:energy_norm}
\interleave w\interleave_h:=\left(\sum_{I_j\in\mathcal{T}_h}\| w_x\|_{I_j}^2+\sum_{x_j\in\mathcal{N}_h}\frac{1}{h} \llbracket w\rrbracket_j^2\right)^{1/2}.
\end{equation}

For non-homogeneous Dirichlet boundary conditions,  penalty terms acting only on the boundary are needed to ensure the stability \cite{liu2018mixed,fu2025analysis,huang2026penalty}. Following \cite{huang2026penalty}, the semi-discrete variational formulation of \eqref{eq:original_problem_uv} reads: find $u_h, v_h\in V_h^k$ such that, for each time $t$,
\begin{subequations}\label{eq:semi_discrete}
\begin{align}
  (u_{ht},q_h)+B(v_h,q_h)+J(u_h,q_h)=L_1(q_h), \quad \forall q_h\in V_h^k,\label{eq:semi_discretea}\\
 -B(w_h,u_h)+(v_h,w_h)=L_2(w_h), \quad \forall w_h\in V_h^k, \label{eq:semi_discreteb}
\end{align}
\end{subequations}
where
\begin{align*}
L_1(q)&=\int_a^b f(x,t)q\,\mathrm{d}x+\frac{ \beta_1}{h}g_a(t)q_x(a)+\frac{ \beta_1}{h}g_b(t)q_x(b),\\
L_2(w)&=f_b(t) w_x(b)-f_a(t) w_x(a)-g_b(t)w(b)+g_a(t)w(a),\\
B(w,q)&=\sum_{I_j\in \mathcal{T}_h}\int_{I_j} w_x q_x \mathrm{d} x
-\sum_{x_j\in\mathcal{N}_h}\llkh w_x\rrkh_j \llbracket q\rrbracket_j-\sum_{x_j\in\mathcal{N}_h^i}\llbracket w\rrbracket_j \llkh q_x\rrkh_j,\\
J(w,q)&=\frac{\beta_1}{h}w_x(a)q_x(a)+\frac{\beta_1}{h}w_x(b)q_x(b).
\end{align*}


\begin{theorem}[$L^2$-stability]\label{thm:L2_stability}
For $\beta_1> 0$, the numerical solution satisfies
\begin{equation}\label{L2_stable}
\|u_h-\tilde{u}_h\|\leq \|u_0-\tilde{u}_0\|,
\end{equation}
where $u_h$ and $\tilde{u}_h$ solve \eqref{eq:semi_discrete} for the same boundary data and source term, but with initial data $u_0$ and $\tilde{u}_0$, respectively.
\end{theorem}

\begin{proof}
Let $v_h$ and $\tilde{v}_h$ denote the auxiliary variables associated with $u_h$ and $\tilde{u}_h$, respectively, and set $e_h=u_h-\tilde{u}_h$, $e_h^*=v_h-\tilde{v}_h$. Since
$u_h$ and $\tilde{u}_h$ have the same boundary data and source term, then $L_2(w_h)=0$ and $L_1(q_h)=0$.
Subtracting the two semi-discrete systems yields
\begin{subequations}\label{eq:semi_discrete_eh}
\begin{align}
  (e_{ht},q_h)+B(e_h^*,q_h)+J(e_h,q_h)=0, \quad \forall q_h\in V_h^k,\label{eq:semi_discrete_eha}\\
 -B(w_h,e_h)+(e_h^*,w_h)=0, \quad \forall w_h\in V_h^k. \label{eq:semi_discrete_ehb}
\end{align}
\end{subequations}
Taking $q_h=e_h$ in \eqref{eq:semi_discrete_eha} and $w_h=e_h^*$ in \eqref{eq:semi_discrete_ehb}, and adding the resulting equations, we obtain
\[
\frac{1}{2}\frac{\mathrm{d}}{\mathrm{d}t}\|e_h\|^2=-\|e_h^*\|^2-J(e_h,e_h)\leq 0,
\]
which implies \eqref{L2_stable}.
\end{proof}

To derive the error estimate, we introduce a global projection operator as in \cite{fu2025analysis,liu2018mixed}. Given $w\in H^2(\mathcal{T}_h)$ and $k\ge 1$, the projection $\Pi_hw\in V_h^k$ is defined by
\begin{subequations}\label{eq:projection}
\begin{align}
& \int_{I_j}(w-\Pi_hw)q_h\mathrm{d} x=0,\quad \forall q_h\in P_{k-2}(I_j),\ j=1,\ldots,M,\label{eq:projection_a}\\
&\llkh\Pi_hw\rrkh_j=\llkh w\rrkh_j,\quad j=1,\ldots,M-1, \label{eq:projection_b}\\
&\llkh(\Pi_hw)_x\rrkh_j=\llkh w_x\rrkh_j,\quad j=0,\ldots,M. \label{eq:projection_c}
\end{align}
\end{subequations}

From \cite{fu2025analysis,liu2018mixed} and a direct calculation, we have the following results.


\begin{lemma}\label{lem:projection_approx}
For any $k\ge 1$, the projection $\Pi_h$ defined by \eqref{eq:projection} exists uniquely { and satisfies
\begin{equation}\label{eq:projection_orth}
B(q_h, w-\Pi_hw)=0,\quad \forall q_h\in V_h^k.
\end{equation}}
 Moreover, for $w\in H^{k+1}(\mathcal{T}_h)$,
\begin{equation}
\|w-\Pi_hw\|+h^s\|w-\Pi_hu\|_s+h^{1/2}\|w-\Pi_hw\|_{\mathcal{N}_h}\leq C_{pr} h^{k+1}\|u\|_{k+1},
\end{equation}
where $\|w\|_{\mathcal{N}_h}:=\left(\sum_{j=1}^N(w^-_j)^2+(w^+_{j-1})^2\right)^{1/2}$, $1\leq s\leq k$ is an integer, and the positive constant $C_{pr}$ is independent of $h$ and $w$.
\end{lemma}

\begin{theorem}\label{thm:semi_error}{ 
Let $u$ be the exact solution of \eqref{eq:original_problem_1D} with $v=-u_{xx}$, and let $(u_h,v_h)$ be the solution of the semi-discrete scheme \eqref{eq:semi_discrete}. 
Assume that $u\in L^{\infty}(0,T;H^{k+3}(\mathcal{T}_h))$, $u_t \in L^{\infty}(0,T;H^{k+1}(\mathcal{T}_h))$ and $u_0 \in H^{k+1}(\mathcal{T}_h)$, and that the parameter satisfies $\beta_1>0$. 
Then, for $k\ge 1$ and any $t \in (0,T]$, there exists a positive constant $C_{sp}$ such that
\begin{equation}\label{eq:semi_error_main}
\|u_h(t)-u(t)\| + \int_0^t \|v(s)-v_h(s)\|\mathrm{d}s \leq C_{sp}h^{k+1},
\end{equation}}
where $C_{sp}$  is  independent of $h$.
\end{theorem}

\begin{proof}
It is easy to verify that the exact solution $u$ and $v=-u_{xx}$ satisfies the following equation
\begin{subequations}\label{eq:exact_variation}
\begin{align}
  (u_t, q_h) + B(v, q_h) + J(u, q_h) &= { L_1(q_h)}, \quad \forall q_h \in V_h^k, \label{eq:exact_variationa} \\
  -B(w_h, u) + (v, w_h) &= { L_2(w_h)}, \quad \forall w_h \in V_h^k. \label{eq:exact_variationb}
\end{align}
\end{subequations}
Let $e_u = u - u_h$ and $e_v = v - v_h$. By \eqref{eq:exact_variation} and \eqref{eq:semi_discrete}, we have the error equations
\begin{subequations}\label{eq:semi_discrete_orth}
\begin{align}
  (e_{ut}, q_h) + B(e_v, q_h) + J(e_u, q_h) &= 0, \quad \forall q_h \in V_h^k, \label{eq:semi_discrete_ortha} \\
  -B(w_h, e_u) + (e_v, w_h) &= 0, \quad \forall w_h \in V_h^k. \label{eq:semi_discrete_orthb}
\end{align}
\end{subequations}
We decompose the errors as
$$
e_u = (u - \Pi_h u) - (u_h - \Pi_h u) =: \eta_u - \xi_u, \qquad
e_v = (v - \Pi_h v) - (v_h - \Pi_h v) =: \eta_v - \xi_v.
$$
Taking $q_h = \xi_u$ and $w_h = \xi_v$ in \eqref{eq:semi_discrete_ortha} and \eqref{eq:semi_discrete_orthb}, respectively, and adding the two resulting equations, we obtain
\begin{equation}\label{eq:err1}
(\xi_{ut}, \xi_u) + (\xi_v, \xi_v) + J(\xi_u, \xi_u)
= (\eta_{ut}, \xi_u) + (\eta_v, \xi_v) + J(\eta_u, \xi_u) + B(\eta_v, \xi_u) - B(\xi_v, \eta_u).
\end{equation}
Using the definition of $\Pi_h$ in \eqref{eq:projection}, integration by parts, and a straightforward calculation, we obtain
\begin{equation}\label{eq:Pih1}
J(\eta_u, \xi_u) = 0, \qquad
B(\eta_v, \xi_u) = \eta_v(x_N)\xi_{ux}(x_N) - \eta_v(x_0)\xi_{ux}(x_0), \qquad
B(\xi_v, \eta_u) = 0.
\end{equation}
Applying Young's inequality yields
\begin{equation}\label{eq:err2}
\begin{aligned}
\eta_v(x_N)\xi_{ux}(x_N) &\leq \frac{h}{2\beta_1}\bigl(\eta_v(x_N)\bigr)^2 + \frac{\beta_1}{2h}\bigl(\xi_{ux}(x_N)\bigr)^2, \\
-\eta_v(x_0)\xi_{ux}(x_0) &\leq \frac{h}{2\beta_1}\bigl(\eta_v(x_0)\bigr)^2 + \frac{\beta_1}{2h}\bigl(\xi_{ux}(x_0)\bigr)^2.
\end{aligned}
\end{equation}
Consequently, by the approximation property of the projection in Lemma~\ref{lem:projection_approx}, we have
{ 
\begin{equation}\label{eq:err3}
h^{1/2}|\eta_v(x_N)| +
h^{1/2}|\eta_v(x_0)| \leq C_{pr} h^{k+1}\|v\|_{k+1}.
\end{equation}
Combining \eqref{eq:err1}--\eqref{eq:err3} with the Cauchy--Schwarz inequality and Lemma~\ref{lem:projection_approx}, we obtain
\begin{equation}\label{eq:err4}
\frac{1}{2}\frac{\mathrm{d}}{\mathrm{d}t}\|\xi_u\|^2 + \|\xi_v\|^2
\leq C_{pr} h^{k+1} \|\xi_u\|\|u_t\|_{k+1} + C_{pr} h^{k+1} \|\xi_v\|\|v\|_{k+1}+\frac{C_{pr}^2}{2\beta_1}h^{2k+2}\|v\|_{k+1}^2.
\end{equation}
Applying Young's inequality again, we obtain
\begin{equation}\label{eq:err5}
\frac{1}{2}\frac{\mathrm{d}}{\mathrm{d}t}\|\xi_u\|^2 + \frac{1}{2}\|\xi_v\|^2
\leq C_{pr} h^{k+1} \|\xi_u\|\|u_t\|_{k+1} + C_{pr}^2 h^{2k+2} \|v\|_{k+1}^2+\frac{C_{pr}^2}{2\beta_1}h^{2k+2}\|u\|_{k+1}^2.
\end{equation}
}
%
Next, following the argument in the proof of Theorem 2.2 in \cite{liu2018mixed}, we arrive at
\begin{equation}\label{eq:err5}
\|\xi_u\| + \int_0^t \|\xi_v\|\,\mathrm{d}s \leq C h^{k+1}.
\end{equation}
Finally, applying Lemma~\ref{lem:projection_approx} and the triangle inequality yields the desired estimate \eqref{eq:semi_error_main}. This completes the proof.
\end{proof}

{ 
\begin{remark}
The global projection operator $\Pi_h$ and the penalty term $J(\cdot,\cdot)$ play a crucial role in the proof of Theorem~\ref{thm:semi_error} for achieving the optimal $(k+1)$-th order accuracy. Specifically, the proof shows that $J(\cdot,\cdot)$ is essential and that only the boundary penalty terms of $u_x$ are necessary. Moreover, if interior penalty terms of $u_x$ are also included in $J(\cdot,\cdot)$, one can similarly obtain $J(\eta_u,\xi_v)=0$ and the $(k+1)$-th order error estimate.
\end{remark}}

\subsection{Fully discrete scheme}

In this section, we apply the LBRFD method \eqref{eq:FD_form}, introduced in Subsection~\ref{sec:LBRIFD}, to discretize the temporal derivative in the semi-discrete problem \eqref{eq:semi_discrete}.

For notational convenience, we omit the subscript $n$ and write $d_j:=d_{n(n-j)}$. Then \eqref{eq:FD_form} becomes
\begin{equation}\label{eq:FD_form3}
\sum_{j=0}^nd_jy_{m-j}=\tau g(y_m),\quad m=n+1,\ldots,N.
\end{equation}

Substituting $t=t_m$ into \eqref{eq:semi_discrete} and replacing the time derivative by the LBRFD formula \eqref{eq:FD_form3}, we obtain, for all $(q_h,w_h)\in V_h^k\times V_h^k$,
\begin{subequations}\label{eq:full_discrete0}
\begin{align}
 \frac{1}{\tau}\sum_{j=0}^nd_j (u_h(t_{m-j}),q_h)+ B(v_h(t_m),q_h)+ J(u_h(t_m),q_h)&= L_1(q_h)(t_m)+ (R^m,q_h), \label{eq:full_discrete0a}\\
 -B(w_h,u_h(t_m))+(v_h(t_m),w_h)&=L_2(w_h)(t_m),  \label{eq:full_discrete0b}
\end{align}
\end{subequations}
where $R^m$ is the truncation error in replacing $u_{h}(\cdot,t_m)$ by the value at $t_m$ of the time-dependent linear barycentric rational interpolant of $u_h$ and denote  $u_h(t_m):=u_h(x,t_m)$ and $v_h(t_m):=v_h(x,t_m)$. Neglecting the truncation error, we arrive at the fully discrete scheme: find $(u_h^m,v_h^m)\in V_h^k\times V_h^k$, $m=n+1,\ldots,N$, such that
\begin{subequations}\label{eq:full_discrete}
\begin{align}
 \frac{1}{\tau}\sum_{j=0}^nd_j (u_h^{m-j},q_h)+ B(v_h^m,q_h)+ J(u_h^m,q_h)&= L_1(q_h)(t_m), \quad \forall q_h\in V_h^k, \label{eq:full_discretea}\\
 -B(w_h,u_h^m)+(v_h^m,w_h)&=L_2(w_h)(t_m), \quad \forall w_h\in V_h^k. \label{eq:full_discreteb}
\end{align}
\end{subequations}
The starting values { $u_h^0, u_h^1,v_h^1,\ldots,u_h^n,v_h^n$} are computed by the startup procedure
\begin{subequations}\label{eq:full_discrete_st}
\begin{align}
\frac{1}{\tau}\sum_{j=0}^nd_{mj}(u_h^j,q_h)+ B(v_h^m,q_h)+ J(u_h^m,q_h)&= L_1(q_h)(t_m), \quad \forall q_h\in V_h^k,\label{eq:full_discrete_sta}\\
-B(w_h,u_h^m)+(v_h^m,w_h)&=L_2(w_h)(t_m), \quad\forall w_h\in V_h^k,\label{eq:full_discrete_stb}\\
u_h^0&=\Pi_h u_0,&& \label{eq:full_discrete_stc}
\end{align}
\end{subequations}
for $m=1,2,\ldots,n$.

The solution procedure \eqref{eq:full_discrete} with the startup procedure \eqref{eq:full_discrete_st} is summarized in Algorithm~\ref{alg:solution_procedure}.

\begin{algorithm}[htbp]
\caption{Solution procedure for the fully discrete scheme}
\label{alg:solution_procedure}
\begin{algorithmic}[1]
\State Set $u_h^0$ from \eqref{eq:full_discrete_stc};
\State Solve the startup linear system formed by \eqref{eq:full_discrete_sta}--\eqref{eq:full_discrete_stb} for $m=1,\ldots,n$
\State \hspace{0.5em} (a system of $nN_{\mathrm{dof}}$ equations in $nN_{\mathrm{dof}}$ unknowns, where $N_{\mathrm{dof}}=(k+1)M$)
\State \textbf{Obtain} $u_h^1,\ldots,u_h^n$ and $v_h^1,\ldots,v_h^n$;
\For{$m=n+1,n+2,\ldots,N$}
    \State Advance $(u_h^m,v_h^m)$ by solving \eqref{eq:full_discrete};
\EndFor
\end{algorithmic}
\end{algorithm}

\begin{remark}\label{rem:startup_alternative}
When $n$ is large, the startup linear system formed by \eqref{eq:full_discrete_st} becomes very large, since it contains $nN_{\mathrm{dof}}$ equations in $nN_{\mathrm{dof}}$ unknowns. In this case, the starting values $(u_h^m,v_h^m),m=0,1,\ldots,n$ may instead be computed by alternative time integrators, such as the three-stage, fifth-order Radau IIA Runge--Kutta method. In the numerical examples below, we compare the performance of several such startup strategies.
\end{remark}

%


\subsection{Stability analysis}

We now analyze the stability of the fully discrete problem \eqref{eq:full_discrete} with the startup procedure \eqref{eq:full_discrete_st}.  For simplicity, we assume homogeneous boundary data, so that
$$L_1(q_h)(t_m)=\int_a^b f(x,t_m)q_h\mathrm{d}x,\quad L_2(w_h)=0.$$
 We write $f^m(x):=f(x,t_m)$ and denote by $\eta_n$ the parameter in Table~\ref{tab:LBRFDM_BDF_eta} associated with the chosen pair $(n,d)$. Throughout this section, we consider only the pairs $(n,d)$ listed in Table~\ref{tab:LBRFDM_BDF_eta} for which the corresponding parameter $\eta_n$ lies in {  $[0,1)$}.

{ 
\begin{lemma}
Let $\beta_1>0$.
Assume that \(u_h^m\), for \(m=1,\ldots,n\), is the solution of the startup procedure \eqref{eq:full_discrete_st}. Then, there exists a positive constant \(C_{ss}\) such that
\begin{equation}\label{eq:startup_stability}
\tau\|v_h^m\|^2+\tau J(u_h^m,u_h^m)\leq C_{ss}\left(\tau\|f^m\|^2+\max_{0\leq j\leq n}\|u_h^j\|^2\right), \quad m=1,\ldots,n.
\end{equation}
\end{lemma}

\begin{proof}
In \eqref{eq:full_discrete_st}, setting \(w_h=v_h^m\) and \(q_h=u_h^m\), and summing the two equations, we obtain
\begin{equation*} 
\tau\|v_h^m\|^2+\tau J(u_h^m,u_h^m)=\tau (f^m,u_h^m)-\sum_{j=0}^n d_{mj}(u_h^j,u_h^m).
\end{equation*}
Applying the Cauchy-Schwarz inequality, along with \(d_{max} = \max_{m=1,\dots,n, \, 0\le j \le n} |d_{mj}|\), we have
\begin{align*}
\tau\|v_h^m\|^2+\tau J(u_h^m,u_h^m) &\leq \frac{1}{2}\tau \|f^m\|^2 + \frac{1}{2}\tau \|u_h^m\|^2 + d_{\max}\sum_{j=0}^n \left( \frac{1}{2}\|u_h^j\|^2 + \frac{1}{2}\|u_h^m\|^2 \right) \\
&\leq \frac{1}{2}\tau \|f^m\|^2 + ((n+1)d_{max}+\frac{\tau}{2}) \max_{0\leq j\leq n}\|u_h^j\|^2.
\end{align*}
 This completes the proof.
\end{proof}

\begin{theorem}\label{thm:full_stable}
Let the pairs \((n,d)\) with \(\eta_n \in [0,1)\) from Table~\ref{tab:LBRFDM_BDF_eta} and $\beta_1>0$.
Assume  $u_h^m, m=0,\ldots,N$ is the solution of the fully discrete scheme \eqref{eq:full_discrete} with the startup procedure \eqref{eq:full_discrete_st}.
Then, for $0<\tau<\tau_0:=2\lambda_0/(1+\eta_n^2)$,
there exists a positive constant
$C_{st}$
independent of $h$ and $\tau$, such that
\begin{equation}\label{eq:full_stable}
\lambda_0\sum_{i=m-n+1}^m\|u_h^i\|^2+\tau \|v_h^m\|^2+\tau J(u_h^m,u_h^m)\leq C_{st}\left(\tau\sum_{j=1}^m\|f^j\|^2+\max_{0\leq j\leq n}\|u_h^j\|^2\right), m\ge n+1.
\end{equation}
where $\lambda_0$ is the minimum eigenvalue of the symmetric positive matrix  $G$ in Lemma~\ref{lem:g_stable}.
\end{theorem}
}

\begin{proof}
In \eqref{eq:full_discretea}, take $q_h=u_h^m-\eta_n u_h^{m-1}$. This yields
\begin{equation}\label{eq:stable1}
\frac{1}{\tau}\sum_{j=0}^nd_j (u_h^{m-j},u_h^m-\eta_n u_h^{m-1})+ B(v_h^m,u_h^m-\eta_n u_h^{m-1})+ J(u_h^m,u_h^m-\eta_n u_h^{m-1})= (f^m,u_h^m-\eta_n u_h^{m-1}).
\end{equation}
In \eqref{eq:full_discreteb}, take $w_h=v_h^m$ at the levels $m$ and $m-1$. A straightforward combination of the two resulting equations gives
\begin{equation}\label{eq:B_identity}
B(v_h^m,u_h^m-\eta_n u_h^{m-1})=(v_h^m-\eta_n v_h^{m-1},v_h^m).
\end{equation}
Substituting \eqref{eq:B_identity} into \eqref{eq:stable1}, we obtain
\begin{equation}\label{eq:stable2}
\frac{1}{\tau}\sum_{j=0}^nd_j (u_h^{m-j},u_h^m-\eta_n u_h^{m-1})+ (v_h^m-\eta_n v_h^{m-1},v_h^m)+J(u_h^m,u_h^m-\eta_n u_h^{m-1})= (f^m,u_h^m-\eta_n u_h^{m-1}).
\end{equation}
Since \(\eta_n\in [0,1)\), it follows from Lemma~\ref{lem:g_stable} that there exists a symmetric positive definite matrix \(G=(g_{ij})_{1\le i,j\le n}\) such that
\begin{align*}
\left(\sum_{j=0}^nd_j u_h^{m-j},u_h^m-\eta_n u_h^{m-1}\right)
=\sum_{i,j=1}^ng_{ij}(u_h^{m-i+1},u_h^{m-j+1})-\sum_{i,j=1}^ng_{ij}(u_h^{m-i},u_h^{m-j})\\
+(\sum_{i=0}^nr_iu_h^{m-i},\sum_{i=0}^nr_iu_h^{m-i})
\geq \sum_{i,j=1}^ng_{ij}(u_h^{m-i+1},u_h^{m-j+1})-\sum_{i,j=1}^ng_{ij}(u_h^{m-i},u_h^{m-j}).
\end{align*}
Define the $G$-norm
\begin{equation}\label{eq:G_energy}
|U^m|_G^2:=\sum_{i,j=1}^ng_{ij}(u_h^{m-i+1},u_h^{m-j+1}).
\end{equation}
Denote by $\lambda_0$ and $\lambda_1$ the smallest and largest eigenvalues of $G$, respectively. Then
\begin{equation}\label{eq:G_positive}
\lambda_0\sum_{i=1}^n\|u_h^{m-i+1}\|^2\leq |U^m|_G^2\leq \lambda_1\sum_{i=1}^n\|u_h^{m-i+1}\|^2.
\end{equation}
Using the identity
\[
(v_h^m-\eta_n v_h^{m-1},v_h^m)=\frac12\|v_h^m\|^2+\frac12\|v_h^m-\eta_n v_h^{m-1}\|^2-\frac{\eta_n^2}{2}\|v_h^{m-1}\|^2
\geq \frac12\|v_h^m\|^2-\frac{\eta_n^2}{2}\|v_h^{m-1}\|^2,
\]
the analogous estimate for $J(\cdot,\cdot)$, and  Cauchy–Schwarz inequality
\[
(f^m,u_h^m-\eta_n u_h^{m-1})\leq \|f^m\|^2+\frac12\|u_h^m\|^2+\frac{\eta_n^2}{2}\|u_h^{m-1}\|^2,
\]
we deduce from \eqref{eq:stable2} { and $\eta_n\in [0,1)$} that, for each $m\ge n+1$,
\begin{align}\label{eq:stable_energy}
|U^m|_G^2+\frac{\tau}{2}\|v_h^m\|^2+\frac{\tau}{2}J(u_h^m,u_h^m)
&\leq |U^{m-1}|_G^2+\frac{\tau}{2}\|v_h^{m-1}\|^2+\frac{\tau}{2}J(u_h^{m-1},u_h^{m-1})\notag\\
&\quad+\tau\left(\|f^m\|^2+\frac12\|u_h^m\|^2+\frac{\eta_n^2}{2}\|u_h^{m-1}\|^2\right).
\end{align}
Summing \eqref{eq:stable_energy} from $m=n+1$ to $m$, we obtain
\begin{align*}
|U^m|_G^2+\frac{\tau}{2}\|v_h^m\|^2+\frac{\tau}{2}J(u_h^m,u_h^m)
&\leq |U^n|_G^2+\frac{\tau}{2}\|v_h^n\|^2+\frac{\tau}{2}J(u_h^n,u_h^n)\\
&\quad+\tau\sum_{j=n+1}^m\left(\|f^j\|^2+\frac12\|u_h^j\|^2+\frac{\eta_n^2}{2}\|u_h^{j-1}\|^2\right).
\end{align*}
Since
\[
\sum_{j=n+1}^m\left(\|u_h^j\|^2+\eta_n^2\|u_h^{j-1}\|^2\right)
\leq (1+\eta_n^2)\sum_{j=n}^{m}\|u_h^j\|^2
\leq \frac{1+\eta_n^2}{\lambda_0}\sum_{\ell=0}^{\lfloor (m-n)/n\rfloor}|U^{m-\ell n}|_G^2,
\]
where $\lfloor\cdot\rfloor$ denotes the floor function, it follows that
\begin{align*}
|U^m|_G^2+\frac{\tau}{2}\|v_h^m\|^2+\frac{\tau}{2}J(u_h^m,u_h^m)
&\leq |U^n|_G^2+\frac{\tau}{2}\|v_h^n\|^2+\frac{\tau}{2}J(u_h^n,u_h^n)\\
&\quad+\tau\sum_{j=n+1}^m\|f^j\|^2
+\frac{\tau(1+\eta_n^2)}{2\lambda_0}\sum_{\ell=0}^{\lfloor (m-n)/n\rfloor}|U^{m-\ell n}|_G^2.
\end{align*}
If $\tau<\tau_0:=2\lambda_0/(1+\eta_n^2)$,
applying the discrete Gr\"onwall inequality yields
\begin{equation}\label{eq:gronwall}
|U^m|_G^2+\frac{\tau}{2}\|v_h^m\|^2+\frac{\tau}{2}J(u_h^m,u_h^m)
\leq C_e\left(|U^n|_G^2+\frac{\tau}{2}\|v_h^n\|^2+\frac{\tau}{2}J(u_h^n,u_h^n)
+\tau\sum_{j=n+1}^m\|f^j\|^2\right),
\end{equation}
where
\begin{equation}\label{eq:Ce_def}
C_e=\exp\!\left(\frac{(1+\eta_n^2)t_m}{2\lambda_0}\right).
\end{equation}
Combining \eqref{eq:gronwall} with \eqref{eq:G_positive}, we obtain
\begin{align*}
\lambda_0\sum_{i=1}^n\|u_h^{m-i+1}\|^2 + \frac{\tau}{2}\|v_h^m\|^2 + \frac{\tau}{2}J(u_h^m, u_h^m)
&\leq C_e\Bigg( \lambda_1\sum_{i=1}^n\|u_h^{n-i+1}\|^2 + \frac{\tau}{2}\|v_h^n\|^2 \\
&\qquad + \frac{\tau}{2}J(u_h^n, u_h^n)
+ \tau\sum_{j=n+1}^m\|f^j\|^2 \Bigg).
\end{align*}
In particular, by \eqref{eq:startup_stability}, the right-hand side depends only on the startup values $u_h^0, \ldots, u_h^n$ and the source term. Consequently, we obtain \eqref{eq:full_stable}. This completes the proof.
\end{proof}


\subsection{Error analysis}

We now derive an error estimate for the fully discrete scheme \eqref{eq:full_discrete}--\eqref{eq:full_discrete_st}. Following the G-stability framework of \cite{lubich2013backward,xu2019backward}, we derive the estimate directly from \eqref{eq:full_discrete} by means of Lemma~\ref{lem:g_stable}; this variational argument is equivalent to the matrix--vector proof in \cite{lubich2013backward,xu2019backward}.

Let $u(t)$ and $v(t)=-u_{xx}(t)$ denote the exact solution of \eqref{eq:original_problem_1D} and let $u^m:=u(x,t_m), v^m:=v(x,t_m)$. Denote by $u_h^m$ and $v_h^m$ the fully discrete approximations at $t_m$.
Let $e_u^m = u_h^m-u^m$ and $e_v^m = v_h^m-v^m$. By  the fully discrete scheme \eqref{eq:full_discrete}--\eqref{eq:full_discrete_st}, \eqref{eq:exact_variation} and \eqref{eq:semi_discrete}, we have the error equations, for any $q_h, w_h\in V_h^k$,
\begin{subequations}\label{eq:full_discrete_err}
\begin{align}
  \frac{1}{\tau}\sum_{j=0}^nd_j (e_u^{m-j}, q_h) + B(e_v^m, q_h) + J(e_u^m, q_h) &= (\delta^m,q_h), \label{eq:full_discrete_erra} \\
  -B(w_h, e_u^m) + (e_v^m, w_h) &= 0. \label{eq:full_discrete_errb}
\end{align}
\end{subequations}
where
\begin{equation}\label{eq:temporal_defect}
\delta^m(x):=\frac{1}{\tau}\sum_{j=0}^n d_ju(x,t_{m-j})-u_{t}(x,t_m),\qquad m=n+1,\ldots,N.
\end{equation}

Let $p$ denote the temporal order of the LBRFD scheme \eqref{eq:FD_form} with the startup procedure \eqref{eq:full_discrete_st}, i.e.,
\begin{equation}\label{eq:temporal_order}
p=\left\{\begin{array}{ll}
d, & \text{if } n-d \text{ is even},\\
d+1, & \text{if } n-d \text{ is odd}.
\end{array}\right.
\end{equation}
For the exact solution $u(t)$ of \eqref{eq:original_problem_1D}, define the temporal defect

For the startup steps $m=1,\ldots,n$, the startup temporal defect $\delta^m$ is defined analogously from the startup procedure \eqref{eq:full_discrete_st}.

\begin{lemma}[Temporal consistency]\label{lem:temporal_consistency}
{  Let $u(x,t)$ be the solution of the original problem \eqref{eq:original_problem_1D} and suppose that $\partial_t^{p+1}u\in L^{\infty}(0,T;L^2(\Omega))$.
Then the defect \eqref{eq:temporal_defect}, as well as the startup defect for $m=1,\ldots,n$, satisfies
\begin{equation}\label{eq:defect_bound}
\|\delta^m\|\leq C_\delta\,\tau^{p},\qquad m=1,\ldots,N,
\end{equation}
where $
C_\delta$ is independent of $\tau$.}
\end{lemma}

\begin{proof}
Since the LBRFD scheme \eqref{eq:FD_form} together with the startup procedure \eqref{eq:FD_form_st} is of order $p$ for ODEs \cite{abdi2019adaptive} and Theorem \ref{thm:derivative_error}, applying them to the semi-discrete evolution yields
\[
\delta^m=\frac{1}{\tau}\sum_{j=0}^n d_{mj}\,u(x,t_j)-u_t(x,t_m)=O(\tau^p),\quad m=1,\ldots,n,
\]
and
\[
\delta^m=\frac{1}{\tau}\sum_{j=0}^n d_j\,u(x,t_{m-j})-u_t(x,t_m)=O(\tau^p),\quad m=n+1,\ldots,N.
\]
in the sense of $L^2(\Omega)$, uniformly in $x\in(a,b)$. Equivalently, $\|\delta^m\|=O(\tau^p)\|\partial_t^{p+1}u\|$. 
\end{proof}

{ 
\begin{theorem}[Fully discrete error estimate]\label{thm:full_error}
Let the pairs \((n,d)\) with \(\eta_n \in [0,1)\) be chosen from Table~\ref{tab:LBRFDM_BDF_eta}, and let $\beta_1>0$.
Let $u$ be the exact solution of \eqref{eq:original_problem_1D} and let $(u_h^m,v_h^m)$ be the fully discrete solution of \eqref{eq:full_discrete} with the startup procedure \eqref{eq:full_discrete_st}. 
Assume that 
$u\in L^{\infty}(0,T;H^{k+3}(\mathcal{T}_h))$ and $\p_t^{p+1} u\in L^{\infty}(0,T;L^2(\Omega))$.
Suppose that the startup values satisfy
\begin{equation}\label{eq:startup_error}
\max_{1\leq m\leq n}\|u^m-u_h^m\|\leq C_{st}(\tau^p+h^{k+1}),
\end{equation}
where $C_{st}$ is independent of $h$ and $\tau$. Let $\lambda_0$ be the minimum eigenvalue of the symmetric positive matrix $G$ in Lemma~\ref{lem:g_stable}.
 Then, for $\tau<\tau_0:=2\lambda_0/(1+\eta_n^2)$, there exists a positive constant $C_{\mathrm{fd}}$ such that
\begin{equation}\label{eq:full_error_main}
\lambda_0\sum_{i=1}^n\|u(t_m)-u_h^m\|^2+\frac{\tau\eta_n^2}{2}\|v(t_m)-v_h^m\|^2\leq C_{fd}(h^{k+1}+\tau^p),\quad m=n+1,\ldots,N,
\end{equation}
where $C_{fd}$ depends on $\sup_{t\in[0,T]}\|u\|_{k+3}$ and $\sup_{t\in[0,T]}\|\partial_t^{p+1}u\|$, but is independent of $h$ and $\tau$.
\end{theorem}

 \begin{proof}
We decompose the total error into two parts
\begin{equation*}
\begin{array}{c}
  e_u^m=u_h^m-\Pi_h u^m+\Pi_h u^m-u^m:=\theta^m_u-\rho^m_u,\qquad m=1,\ldots,N,  \\[6pt]
  e_v^m=v_h^m-\Pi_h v^m+\Pi_h v^m-v^m:=\theta^m_v-\rho^m_v,\qquad m=1,\ldots,N.
\end{array}
\end{equation*}
It follows from \eqref{eq:full_discrete_stc} that $\theta^0_u=\Pi_h u_0-u_h(0)=0$. Using the definition of $\Pi_h$ in \eqref{eq:projection} and Lemma~\ref{lem:projection_approx}, we deduce
\[
B(w_h,\rho_u^m)=0,\quad J(\rho_u^m,q_h)=0,\quad B(\rho^m_v,q_h)= \rho^m_v(x_N)q_{hx}(x_N) - \rho^m_v(x_0)q_{hx}(x_0).
\]
With this, the error equation \eqref{eq:full_discrete_err} can be rewritten, for any $q_h, w_h\in V_h^k$, as
\begin{subequations}\label{eq:error_full1}
\begin{align}
\frac{1}{\tau}\sum_{j=0}^n d_j(\theta^{m-j}_u,q_h)+ B(\theta^m_v,q_h)+ J(\theta^m_u,q_h)&=(\delta^m,q_h)+\frac{1}{\tau}\sum_{j=0}^n d_j(\rho^{m-j}_u,q_h)+ B(\rho^m_v,q_h),\label{eq:error_full1_a}\\
-B(w_h,\theta^m_u)+(\theta^m_v,w_h)&=(\rho^m_v,w_h),\label{eq:error_full1_b}
\end{align}
\end{subequations}
for $m=n+1,\ldots,N$.

In \eqref{eq:error_full1}, taking $q_h=\theta^m_u-\eta_n\theta^{m-1}_u$ and $w_h=\theta^m_v$ at the levels $m$ and $m-1$, respectively, and summing the resulting equations, we arrive at
\begin{equation}\label{eq:stable4}
\begin{array}{c}
\d\frac{1}{\tau}\sum_{j=0}^n d_j (\theta^{m-j}_u,\theta_u^m-\eta_n \theta_u^{m-1})+ (\theta^m_v-\eta_n\theta^{m-1}_v,\theta^m_v)+ J(\theta^m_u,\theta_u^m-\eta_n \theta_u^{m-1})
= (\delta^m,\theta_u^m-\eta_n \theta_u^{m-1})\\[6pt]
+\d\frac{1}{\tau}\sum_{j=0}^n d_j(\rho^{m-j}_u,\theta_u^m-\eta_n \theta_u^{m-1})+ B(\rho^m_v,\theta_u^m-\eta_n \theta_u^{m-1})+(\rho^m_v-\eta_n\rho^{m-1}_v,\theta^m_v).
\end{array}
\end{equation}
Applying Lemma~\ref{lem:projection_approx} and Lemma~\ref{lem:temporal_consistency}, we derive
\[ (\delta^m,\theta_u^m-\eta_n \theta_u^{m-1})\leq 2C_{\delta}^2\tau^{2p}+\frac{1}{2}\|\theta_u^m\|^2+\frac{\eta_n^2}{2}\|\theta_u^{m-1}\|^2,\]
\[
\sum_{j=0}^n d_j(\rho^{m-j}_u,\theta_u^m-\eta_n \theta_u^{m-1})\leq 2C_{\mathrm{pr}}^2\max_{j=0,\ldots,n}|d_j|^2h^{2k+2}\sup_{t\in [0,1]}\|u\|_{k+1}^2+\frac{1}{2}\|\theta_u^m\|^2+\frac{\eta_n^2}{2}\|\theta_u^{m-1}\|^2,
\]
\[
|B(\rho^m_v,\theta_u^m-\eta_n \theta_u^{m-1})|\leq \frac{1}{2}C_{\mathrm{pr}}^2h^{2k+2}\sup_{t\in [0,1]}\|v\|_{k+1}^2+\frac{1}{2}J(\theta^m_u-\eta_n\theta_u^{m-1},\theta^m_u-\eta_n\theta_u^{m-1}),
\]
and Young's inequality gives
\[
(\rho^m_v-\eta_n\rho^{m-1}_v,\theta^m_v)\leq\frac{1}{\epsilon} C_{\mathrm{pr}}^2h^{2k+2}\sup_{t\in [0,1]}\|v\|_{k+1}^2+\frac{\epsilon}{2}\|\theta_v^m\|^2,\quad \forall \epsilon>0.
\]
 A direct calculation  implies that
\[
(\theta^m_v-\eta_n\theta^{m-1}_v,\theta^m_v)\geq \frac{1}{2}\|\theta^m_v\|^2-\frac{\eta_n^2}{2}\|\theta^{m-1}_v\|^2,
\]
and
\[
J(\theta^m_u,\theta_u^m-\eta_n \theta_u^{m-1})= \frac{1}{2} J(\theta^m_u,\theta^m_u)+\frac{1}{2}J(\theta^m_u-\eta_n\theta_u^{m-1},\theta^m_u-\eta_n\theta_u^{m-1})-\frac{\eta_n^2}{2}J(\theta_u^{m-1},\theta_u^{m-1}).
\]
Set
\[
|U^m|_G^2:=\sum_{i,j=1}^n g_{ij}(\theta_u^{m-i+1},\theta_u^{m-j+1}).
\]
Lemma~\ref{lem:g_stable} then yields
\[
\frac{1}{\tau}\sum_{j=0}^n d_j (\theta^{m-j}_u,\theta_u^m-\eta_n \theta_u^{m-1})\geq |U^m|_G^2-|U^{m-1}|_G^2.
\]
Taking $\epsilon=1-\eta_n^2$, using $\eta_n\in [0,1)$, and inserting the above inequalities into \eqref{eq:stable4}, we arrive at
\begin{align}
|U^m|_G^2+\frac{\tau\eta_n^2}{2}\|\theta_v^m\|^2+\frac{\tau}{2}J(\theta_u^m,\theta_u^m)&\leq |U^{m-1}|_G^2+
\tau\|\theta_u^m\|^2+{\eta_n^2}\tau\|\theta_u^{m-1}\|^2\notag\\
&\quad+\frac{\tau\eta_n^2}{2}\|\theta_v^{m-1}\|^2+\frac{\tau}{2}J(\theta_u^{m-1},\theta_u^{m-1})\notag\\
&\quad+C_1\tau^{2p}+C_1h^{2k+2}.
\end{align}
Assume that $\tau<\tau_0=2\lambda_0/(1+\eta_n^2)$. Then, following the proof of Theorem~\ref{thm:full_stable}, we deduce
\begin{align}\label{eq:full_error_st2}
\lambda_0\sum_{i=1}^n\|\theta^{m-i+1}_u\|^2+\frac{\tau\eta_n^2}{2}\|\theta_v^m\|^2+\frac{\tau}{2}J(\theta_u^m,\theta_u^m)&\leq C_{\mathrm{er}}\Bigg(\lambda_1\sum_{i=1}^n\|\theta^{n-i+1}_u\|^2\notag\\
&\quad+\frac{\tau\eta_n^2}{2}\|\theta_v^n\|^2+\frac{\tau}{2}J(\theta_u^n,\theta_u^n)\notag\\
&\quad+C_1\tau^{2p}+C_1h^{2k+2}\Bigg).
\end{align}

We next consider the startup errors $\theta^1,\ldots,\theta^n$. Using the startup procedure \eqref{eq:full_discrete_st} and an argument similar to the above, we arrive at
\begin{align*}
 B(\theta^m_v,q_h)+ J(\theta^m_u,q_h)&=(\delta^m,q_h)+\frac{1}{\tau}\sum_{j=0}^n d_j(e^{m-j}_u,q_h)+ B(\rho^m_v,q_h), \\
-B(w_h,\theta^m_u)+(\theta^m_v,w_h)&=(\rho^m_v,w_h),
\end{align*}
for $m=1,2,\ldots,n$. Taking $q_h=\theta_u^m$ and $w_h=\theta_v^m$ and summing the two equations, we get
\begin{align*}
 \|\theta^m_v\|^2+ J(\theta^m_u,\theta_u^m)=(\delta^m,\theta_u^m)+\frac{1}{\tau}\sum_{j=0}^n d_j(e^{m-j}_u,\theta_u^m)+ B(\rho^m_v,\theta_u^m)+(\rho^m_v,\theta_v^m).
\end{align*}
By an argument similar to the above, it follows that
\[
\tau(\theta^m_v,\theta^m_v)+\tau J(\theta^m_u,\theta^m_u)\leq C\max_{m=1,\ldots,n}\|\theta_u^m\|^2+C_2\tau^{2p}+C_2h^{2k+1}.
\]
Using assumption \eqref{eq:startup_error} and the triangle inequality, we deduce
\[
\|\theta^{m}_u\|^2+\tau\|\theta_v^m\|^2+\tau J(\theta^m_u,\theta^m_u)\leq C_3(\tau^{2p}+h^{2k+1}),\quad m=1,2,\ldots,n.
\]
Combining this with \eqref{eq:full_error_st2}, we conclude
\begin{align}\label{eq:full_error_st3}
\lambda_0\sum_{i=1}^n\|\theta^{m-i+1}_u\|^2+\frac{\tau\eta_n^2}{2}\|\theta_v^m\|^2+\frac{\tau}{2}J(\theta_u^m,\theta_u^m)\leq C_4\tau^{2p}+C_4h^{2k+2},\quad m=n+1,\ldots,N.
\end{align}

Finally, applying Lemma~\ref{lem:projection_approx} and the triangle inequality, we obtain the desired estimate \eqref{eq:full_error_main}.
\end{proof}

}


%


\section{Two-dimensional case}

In this section, we develop a fully discrete scheme for the two-dimensional
linear fourth-order problem \eqref{eq:original_problem_2D} by coupling a linear
barycentric rational  finite difference   method in time
with IPDG method in space.
Since the subsequent stability and error analysis follows the same lines as
in Section~\ref{sec:onedimension}, we state only the main results for the two-dimensional setting.

We consider the model problem
\begin{equation}\label{eq:original_problem_2D}
u_t+\Delta^2 u=f(\mathbf{x},t),\quad (\mathbf{x},t)\in\Omega\times (0,T], \quad u(\mathbf{x},0)=u_0(\mathbf{x}),\quad \mathbf{x}\in\Omega,
\end{equation}
subject to the Dirichlet boundary conditions
\begin{equation}\label{eq:Dirichlet_2D}
u=g_0,\quad \frac{\p u}{\p\mathbf{n}}=g_1,\quad \text{on}\ \ \p\Omega.
\end{equation}

\subsection{Spatial discretization}

Let $\Omega\subset\mathbb{R}^2$ be a bounded polygonal domain, partitioned
into a shape-regular triangulation $\mathcal{T}_h=\{E\}$. For each
$E\in\mathcal{T}_h$, denote by $h_E$ its diameter and set
$h:=\max_{E\in\mathcal{T}_h} h_E$. The edges of the triangulation are
partitioned into the interior edges $\mathcal{E}_h^{\mathrm i}$ and the
boundary edges $\mathcal{E}_h^{\mathrm b}$, and we write
$\mathcal{E}_h:=\mathcal{E}_h^{\mathrm i}\cup\mathcal{E}_h^{\mathrm b}$.
For each edge $e\in\mathcal{E}_h$, let $h_e$ denote its length. On
shape-regular meshes, there exists a constant $C_H>0$, independent of $h$
and $h_e$, such that
\begin{equation}\label{eq:shape_regular}
C_H^{-1}h_e \le h \le C_H h_e,\qquad \forall e\in\mathcal{E}_h.
\end{equation}

We define the broken Sobolev space
\[
H^s(\mathcal{T}_h):=\{\, v\in L^2(\Omega): v|_E\in H^s(E),\ \forall E\in\mathcal{T}_h \,\},
\]
and, for any bounded subdomains $O_1,O_2\subset\Omega$,
\[
H^s(O_1\cup O_2):=\{\, v\in L^2(\Omega): v|_{O_1}\in H^s(O_1),\ v|_{O_2}\in H^s(O_2) \,\}.
\]
The discontinuous finite element space is
\[
V_h^k:=\{\, v\in L^2(\Omega): v|_E\in P_k(E),\ \forall E\in\mathcal{T}_h \,\},
\]
where $P_k(E)$ denotes the space of polynomials of degree at most $k$ on $E$.

 Let $E\in\mathcal{T}_h$ and let $\mathbf{n}_E$ be the outward unit normal to
$\partial E$. Suppose that $E'\in\mathcal{T}_h$ shares a common edge
$e=E\cap E'\in\mathcal{E}_h^{\mathrm i}$ with $E$. For
$w\in H^1(E\cup E')$, the average and jump of $w$ on $e$ are defined by
\begin{equation}\label{eq:average_jump}
\llkh w\rrkh_e:=\frac12\bigl(w|_{E}+w|_{E'}\bigr),\qquad
\llbracket w\rrbracket_e:=w|_{E'}-w|_{E}.
\end{equation}
On a boundary edge $e\in\mathcal{E}_h^{\mathrm b}$, with $E$ the adjacent element,
we set
\[
\llkh w\rrkh_e:=w|_{E},\qquad \llbracket w\rrbracket_e:=-w|_{E}.
\]
For $w\in H^2(\mathcal{T}_h)$, we define the discrete energy norm
\begin{equation}\label{eq:energy_norm_2D}
\interleave w\interleave_h^2:= \sum_{E\in\mathcal{T}_h}\int_E |\nabla w|^2\,\mathrm{d}\mathbf{x}
+\sum_{e\in\mathcal{E}_h}\int_e\frac{\llbracket { w}\rrbracket_e^2}{h_e}\,\mathrm{d}s.
\end{equation}
One verifies that $\interleave\cdot\interleave_h$ defines a norm on
$H^2(\mathcal{T}_h)$; see \cite{suli2007hp}.

 By introducing the auxiliary variable $v=-\Delta u$, problem
\eqref{eq:original_problem_2D} can be rewritten as
\begin{equation}\label{eq:problem_2D_sys}
u_t-\Delta v=f,\qquad v+\Delta u=0,\qquad \mathbf{x}\in\Omega.
\end{equation}

A mixed discontinuous Galerkin method for the biharmonic operator in two
dimensions was proposed in \cite{xiong2017priori}. In this section, we
employ a mixed IPDG discretization following \cite{xiong2017priori}.
The semi-discrete variational formulation reads: find $u_h,v_h\in V_h^k$
such that, for each time $t>0$,
\begin{subequations}\label{eq:semi_discrete_2D}
\begin{align}
  (u_{ht},q_h)+B_2(v_h,q_h)+J_2(u_h,q_h)&=L_3(q_h),
  && \forall q_h\in V_h^k,\label{eq:semi_discretea_2D}\\
 -B_2(w_h,u_h)+(v_h,w_h)&=L_4(w_h),
  && \forall w_h\in V_h^k, \label{eq:semi_discreteb_2D}
\end{align}
\end{subequations}
where

\begin{align*}
B_2(w,q)&:=\sum_{E\in\mathcal{T}_h}\int_E \nabla w\cdot\nabla q\,\mathrm{d}\mathbf{x}
+\sum_{e\in\mathcal{E}_h}\int_e \llkh \frac{\partial w}{\partial\mathbf{n}}\rrkh_e
\llbracket q\rrbracket_e\,\mathrm{d}s
+\sum_{e\in\mathcal{E}_h^{\mathrm i}}\int_e \llkh \frac{\partial q}{\partial\mathbf{n}}\rrkh_e
\llbracket w\rrbracket_e\,\mathrm{d}s
+J_I(w,q),\\
J_2(w,q)&:=\sum_{e\in\mathcal{E}_h^b}\int_e\frac{\beta_0}{h_e}\llbracket w\rrbracket_e\llbracket q\rrbracket_e\mathrm{d}s,\quad
J_I(w,q):=\sum_{e\in\mathcal{E}_h^{\mathrm i}}\int_e\frac{\beta_0}{h_e}\,
\llbracket w\rrbracket_e\llbracket q\rrbracket_e\,\mathrm{d}s,\\
L_3(q)&:=\int_{\Omega} fq\mathrm{d}\mathbf{x}
+\sum_{e\in\mathcal{E}_h^{\mathrm b}}\int_e\frac{\beta_0}{h_e}g_0q\mathrm{d}s,\quad
L_4(w):=\sum_{e\in\mathcal{E}_h^{\mathrm b}}\int_e
{ \Bigl(\frac{\partial w}{\partial\mathbf{n}}\,g_0-g_1\,w\Bigr)}\,\mathrm{d}s.
\end{align*} 
Here $\beta_0>0$ is a penalty parameter, and $\partial w/\partial\mathbf{n}:=\nabla w\cdot\mathbf{n}_E$
on each edge $e\subset\partial E$.

\begin{theorem}[$L^2$-stability]\label{thm:L2_stability_2D}
For $\beta_0\ge 0$, the semi-discrete solution $(u_h,v_h)$ of
\eqref{eq:semi_discrete_2D} satisfies
\begin{equation}\label{eq:L2_stable_2D}
\|u_h(t)-\tilde{u}_h(t)\|\le \|u_0-\tilde{u}_0\|,\qquad t\in[0,T],
\end{equation}
where $(\tilde{u}_h,\tilde{v}_h)$ denotes the solution corresponding to the
same boundary data and source term, but with initial data $\tilde{u}_0$ in
place of $u_0$.
\end{theorem}
\begin{proof}
The proof follows the same argument as that of Theorem~\ref{thm:L2_stability}
and is therefore omitted.
\end{proof}

Following \cite{xiong2017priori}, we introduce an elliptic projection
$\Pi_h:H^2(\mathcal{T}_h)\to V_h^k$ as follows: given $w\in H^2(\mathcal{T}_h)$,
find $\Pi_h w\in V_h^k$ such that
\begin{equation}\label{eq:projection_2D}
\left\{\begin{aligned}
B_3(w-\Pi_h w,q_h)&=0, && \forall q_h\in V_h^k,\\
(w-\Pi_h w,1)&=0,
\end{aligned}\right.
\end{equation}
where
\[
B_3(w,q):=B_I(w,q)+J_I(w,q),
\]
and
\begin{align*}
B_I(w,q):=\sum_{E\in\mathcal{T}_h}\int_E \nabla w\cdot\nabla q\,\mathrm{d}\mathbf{x}
+\sum_{e\in\mathcal{E}_h^{\mathrm i}}\int_e \llkh \frac{\partial w}{\partial\mathbf{n}}\rrkh_e
\llbracket q\rrbracket_e\,\mathrm{d}s+\sum_{e\in\mathcal{E}_h^{\mathrm i}}\int_e \llkh \frac{\partial q}{\partial\mathbf{n}}\rrkh_e
\llbracket w\rrbracket_e\,\mathrm{d}s.
\end{align*}


According to Lemma~5 in \cite{xiong2017priori}, we have the following
approximation properties of the projection $\Pi_h$.

\begin{lemma}[Projection error]\label{lem:projection_error_2D}
Assume that the penalty parameter $\beta_0>0$ is sufficiently large.
Given $w\in H^{k+1}(\mathcal{T}_h)$, there exists a unique function
$\Pi_h w\in V_h^k$ satisfying \eqref{eq:projection_2D}. Moreover,
there exists a positive constant $C_{pd}$, independent of $h$,
such that
\begin{equation}\label{eq:projection_error_H1}
\interleave w-\Pi_h w\interleave_h
\le C_{pd}\sum_{E\in\mathcal{T}_h} h_E^k\|w\|_{k+1,E},
\end{equation}
and
\begin{equation}\label{eq:projection_error_L2}
\|w-\Pi_h w\|
\le C_{pd}\sum_{E\in\mathcal{T}_h} h_E^{k+1}\|w\|_{k+1,E}.
\end{equation}
\end{lemma}



\begin{theorem}\label{thm:semi_error_2D}
{ 
Let $u$ be the exact solution of \eqref{eq:original_problem_2D} with $v=-\Delta u$ and let $u_h,v_h$ denote the semi-discrete solution of \eqref{eq:semi_discrete_2D}. Assume that 
$u\in L^{\infty}(0,T;H^{k+3}(\mathcal{T}_h))$, $u_t\in L^{\infty}(0,T;H^{k+1}(\mathcal{T}_h))$ and  the penalty parameter $\beta_0>0$ is sufficiently large.
Then, for $k\ge 1$, there exists a positive constant
$C_{\mathrm{sp}}$
such that
\begin{equation}\label{eq:semi_error_main_2D}
\|u_h(t)-u(t)\|^2+\int_0^t\|v_h(t)-v(t)\|^2\,\mathrm{d}s
\le  C_{sp}h^{2k-2},
\end{equation}
where $C_{sp}$ depends on $\sup_{t\in [0,T]}\|u\|_{k+3}$ and $\sup_{t\in [0,T]}\|u_t\|_{k+1}$, but independent of $h$.}
\end{theorem}

\begin{proof}
Since $u$ satisfies \eqref{eq:problem_2D_sys} with $v=-\Delta u$, testing the
continuous mixed formulation with $q_h,w_h\in V_h^k$ yields
\begin{align*}
  (u_t,q_h)+B_2(v,q_h)+J_2(u,q_h)&=L_3(q_h), \quad \forall q_h\in V_h^k,\\
 -B_2(w_h,u)+(v,w_h)&=L_4(w_h), \quad \forall w_h\in V_h^k.
\end{align*}
Subtracting \eqref{eq:semi_discrete_2D} from these identities gives
\begin{align*}
  (u_t-u_{ht},q_h)+B_2(v-v_h,q_h)+J_2(u-u_h,q_h)&=0,
  \quad \forall q_h\in V_h^k,\\
 -B_2(w_h,u-u_h)+(v-v_h,w_h)&=0,
 \quad \forall w_h\in V_h^k.
\end{align*}
Set
\[
\xi_u:=\Pi_h u-u_h,\quad \eta_u:=\Pi_h u-u,\qquad
\xi_v:=\Pi_h v-v_h,\quad \eta_v:=\Pi_h v-v.
\]
Taking $q_h=\xi_u$ and $w_h=\xi_v$, we obtain
\begin{equation}\label{eq:energy_error_2D}
(\xi_{ut},\xi_u)+(\xi_v,\xi_v)+J_2(\xi_u,\xi_u)
=(\eta_{ut},\xi_u)+(\eta_v,\xi_v)+J_2(\eta_u,\xi_u)
+B_2(\eta_v,\xi_u)-B_2(\xi_v,\eta_u).
\end{equation}
By the definition \eqref{eq:projection_2D} of $\Pi_h$, and the relation between
$B_2$ and $B_3$, we have
\[
B_2(\eta_v,\xi_u) -B_2(\xi_v,\eta_u)
=\sum_{e\in\mathcal{E}_h^{\mathrm b}}\int_e
\frac{\partial \eta_v}{\partial\mathbf{n}}\xi_u-\frac{\partial \xi_v}{\partial\mathbf{n}}\,\eta_u\,\mathrm{d}s.
\]
Applying Lemma~\ref{lem:projection_error_2D}, together with the inverse and
trace inequalities, we obtain
\begin{align*}
(\eta_{ut},\xi_u)
&\le C_{pd}\sum_{E\in\mathcal{T}_h} h_E^{k+1}\|u_t\|_{k+1,E}\|\xi_u\|,\\
(\eta_v,\xi_v)
&\le C_{pd}\sum_{E\in\mathcal{T}_h} h_E^{k+1}\|v\|_{k+1,E}\|\xi_v\|,\\
J_2(\eta_u,\xi_u)
&\le C_{pd}\sum_{E\in\mathcal{T}_h} h_E^{k}\|u\|_{k+1,E}\,J_2(\xi_u,\xi_u)^{1/2},\\
\sum_{e\in\mathcal{E}_h^{\mathrm b}}\int_e
\frac{\partial\eta_v}{\partial\mathbf{n}}\,\xi_u\,\mathrm{d}s
&\le C_{p}\sum_{E\in\mathcal{T}_h} h_E^{k}\|v\|_{k+1,E}\,J_2(\xi_u,\xi_u)^{1/2},\\
\sum_{e\in\mathcal{E}_h^{\mathrm b}}\int_e
\Bigl(-\frac{\partial \xi_v}{\partial\mathbf{n}}\Bigr)\eta_u\,\mathrm{d}s
&\le C_{p}\sum_{E\in\mathcal{T}_h} h_E^{k-1}\|u\|_{k+1,E}\|\xi_v\|.
\end{align*}
Inserting these estimates into \eqref{eq:energy_error_2D} and applying Young's
inequality yields
\[
(\xi_{ut},\xi_u)+\|\xi_v\|^2
\le C_zh^{2k-2}+\|\xi_u\|^2.
\]
Since $u_h(0)=\Pi_hu_0$, we have $\xi_u(0)=0$. Integrating the
above inequality in time and applying the continuous Gr\"onwall inequality, we
obtain
\[
\|\xi_u(t)\|^2+\int_0^t\|\xi_v(s)\|^2\,\mathrm{d}s
\le  C_{sp}^2h^{2k-2}.
\]
The triangle inequality and Lemma~\ref{lem:projection_error_2D} then give
\eqref{eq:semi_error_main_2D}.
\end{proof}

{ 
\begin{remark}
For the two-dimensional problem, Theorem~\ref{thm:semi_error_2D} shows that the semi-discrete convergence order for $u_h(t)$ is $k-1$, two orders below optimal. This loss stems mainly from the trace and inverse inequalities used to bound the boundary terms $J_2(\eta_u,\xi_u)$ and $B_2(\eta_v,\xi_u)-B_2(\xi_v,\eta_u)$ on $\mathcal{E}_h^b$, the latter arising from the asymmetry of the bilinear form $B_2(\cdot,\cdot)$. As can be seen from the proof of this theorem, a boundary penalty term on $u$ is essential for the analysis, while an interior penalty term on $u$ may also be included.
\end{remark}

\begin{remark}
Our analysis is carried out on shape-regular triangular meshes with piecewise polynomials of total degree at most $k$ in a general bounded polygonal domain $\Omega$. In contrast, \cite{liu2018mixed} established optimal error bounds only for periodic problems on rectangular meshes with tensor-product spaces $Q^k$, and \cite{fu2025analysis} derived optimal error estimates for Dirichlet conditions in space, but again within the $Q^k$ framework on rectangular domains. The optimal estimates in both works rely crucially on the one-dimensional projection $\Pi_h$ defined in \eqref{eq:projection} and on approximation by the tensor-product space $Q^k$ on rectangular elements. However, no such result is available on triangular meshes.
\end{remark}
}

\subsection{Fully discrete scheme}
In this section, we present the fully discrete scheme for the two-dimensional
problem \eqref{eq:original_problem_2D} and state the stability and error
estimates for the approximate solution. Since the analysis follows the same
lines as in Section~\ref{sec:onedimension}, we omit the proofs and record only the main results.

\paragraph{Fully discrete scheme.}
Replacing the temporal derivative in the semi-discrete problem
\eqref{eq:semi_discrete_2D} by the LBRFD formula \eqref{eq:FD_form}, we
obtain: find $(u_h^m,v_h^m)\in V_h^k\times V_h^k$, $m=n+1,\ldots,N$, such that
\begin{subequations}\label{eq:full_discrete_2D}
\begin{align}
 \frac{1}{\tau}\sum_{j=0}^n d_j (u_h^{m-j},q_h)+ B_2(v_h^m,q_h)+ J_2(u_h^m,q_h)
 &=L_3(q_h)(t_m), \quad\forall q_h\in V_h^k, \label{eq:full_discretea_2D}\\
 -B_2(w_h,u_h^m)+(v_h^m,w_h)&=L_4(w_h)(t_m), \quad \forall w_h\in V_h^k. \label{eq:full_discreteb_2D}
\end{align}
\end{subequations}

\paragraph{Startup procedure.}
The starting values $u_h^0,u_h^1,v_h^1,\ldots,u_h^n,v_h^n$ are computed by
\begin{subequations}\label{eq:full_discrete_st_2D}
\begin{align}
\sum_{j=0}^n d_{mj}(u_h^{m-j},q_h)+\tau B_2(v_h^m,q_h)+\tau J_2(u_h^m,q_h)
&=\tau L_3(q_h)(t_m), \quad \forall q_h\in V_h^k,\label{eq:full_discrete_sta_2D}\\
-B_2(w_h,u_h^m)+(v_h^m,w_h)&=L_4(w_h)(t_m), \quad \forall w_h\in V_h^k,\label{eq:full_discrete_stb_2D}\\
u_h^0&=\Pi_h u_0, \quad \label{eq:full_discrete_stc_2D}
\end{align}
\end{subequations}
{ for $m=1,\ldots,n$. }

\begin{theorem}[Stability]\label{thm:full_stable_2D}
Let the pairs \((n,d)\) with \(\eta_n \in [0,1)\) from Table~\ref{tab:LBRFDM_BDF_eta} and $\beta_0>0$.
Assume  $u_h^m, m=0,\ldots,N$ is the solution of the fully discrete scheme \eqref{eq:full_discrete_2D} with the startup procedure \eqref{eq:full_discrete_st_2D}.
Then, for $0<\tau<\tau_0:=2\lambda_0/(1+\eta_n^2)$,
there exists a positive constant
$C_{st}$
independent of $h$ and $\tau$, such that
\begin{equation}\label{eq:full_stable_2D}
\sum_{i=m-n+1}^m\|u_h^i\|^2+\tau \|v_h^m\|^2+\tau J_2(u_h^m,u_h^m)
\le C_{st}\left(\tau\sum_{j=n}^m\|f^j\|^2+\max_{0\le j\le n}\|u_h^j\|^2\right).
\end{equation}
where $\lambda_0$ is the minimum eigenvalue of the symmetric positive matrix  $G$ in Lemma~\ref{lem:g_stable}.
\end{theorem}

 \begin{proof}
The proof is analogous to that of Theorem~\ref{thm:full_stable} and is
therefore omitted.
\end{proof}

%

\begin{theorem}[Fully discrete error estimate]\label{thm:full_error_2D}
Let the pairs \((n,d)\) with \(\eta_n \in [0,1)\) be chosen from Table~\ref{tab:LBRFDM_BDF_eta}, and let $\beta_0>0$.
Let $u$ be the exact solution of \eqref{eq:original_problem_2D} and let $(u_h^m,v_h^m)$ be the fully discrete solution of \eqref{eq:full_discrete_2D}--\eqref{eq:full_discrete_st_2D}. 
Assume that 
$u\in L^{\infty}(0,T;H^{k+3}(\mathcal{T}_h))$ and $\p_t^{p+1} u\in L^{\infty}(0,T;L^2(\Omega))$.
Suppose that the startup values satisfy
\begin{equation}\label{eq:startup_error}
\max_{1\leq m\leq n}\|u^m-u_h^m\|\leq C_{st}(\tau^p+h^{k-1}),
\end{equation}
where $C_{st}$ is independent of $h$ and $\tau$. Let $\lambda_0$ be the minimum eigenvalue of the symmetric positive matrix $G$ in Lemma~\ref{lem:g_stable}.
 Then, for $\tau<\tau_0:=2\lambda_0/(1+\eta_n^2)$, there exists a positive constant $C_{\mathrm{fd}}$ such that
\begin{equation}\label{eq:full_error_main}
\lambda_0\sum_{i=1}^n\|u^m-u_h^m\|^2+\frac{\tau\eta_n^2}{2}\|v^m-v_h^m\|^2\leq C_{fd}(\tau^p+h^{k-1}),\quad m=n+1,\ldots,N,
\end{equation}
where $C_{fd}$ depends on $\sup_{t\in[0,T]}\|u\|_{k+3}$ and $\sup_{t\in[0,T]}\|\partial_t^{p+1}u\|$, but is independent of $h$ and $\tau$.
\end{theorem}

\begin{proof}
Using arguments similar to those in the proofs of Theorem~\ref{thm:full_error} and Theorem~\ref{thm:semi_error_2D}, we obtain the result of this theorem. The detailed proof is omitted.
\end{proof}

\begin{remark}
The theory predicts a convergence order of $k-1$ for $u_h^m$, whereas numerical experiments attain the optimal order $k+1$.
\end{remark}


\section{Numerical examples}

In this section, we present several numerical examples for one- and two-dimensional linear fourth-order parabolic initial-boundary value problems to confirm the theoretical convergence results of the LBRFD method. Both the temporal convergence order in $\tau$ and the spatial convergence order in $h$ are examined. We compare the LBRFD method with the $p$-step BDF method (BDF$p$) for $p=1,\ldots,6$.

For the computation of the starting values $u_h^0, u_h^1,v_h^1, \ldots, u_h^n,v_h^n$, three approaches are employed: the barycentric rational interpolation formula (BRIF) given in \eqref{eq:full_discrete_st} (or \eqref{eq:full_discrete_st_2D} in two dimensions), the 3-stage Radau IIA method of classical order 5 (3s-Radau IIA), and the $L^2$ projection of the exact solution onto $V_h^k$ (denoted as ``Exact solution'').

To show the efficiency of the methods and the accuracy of the obtained approximations, we measure
$$\|e_h\|_{S,L^2}:=\max_{0\leq m\leq n}\|u^m-u_h^m\|,\quad\|e_h\|_{N,L^2}:=\max_{0\leq m\leq N}\|u^m-u_h^m\|,$$
and
$$\|e_h\|_{N,L^\infty}:=\max_{0\leq m\leq n}\|u^m-u_h^m\|_{\infty},\quad\|e_h\|_{N,\interleave\cdot\interleave_h}:=\max_{0\leq m\leq N}\interleave u^m-u_h^m\interleave_h.$$

\begin{example}\label{exm5-1}
Consider the one-dimensional fourth-order problem
\begin{equation}\label{eq:example_1D}
\begin{aligned}
u_t+u_{xxxx} &= f(x,t), &&(x,t)\in [0,2\pi]\times(0,1], \\
u(x,0) &= \sin(x), &&x\in[0,2\pi],
\end{aligned}
\end{equation}
with boundary conditions chosen such that the exact solution is
$$
u(x,t)=e^{-t}\sin(x).
$$
\end{example}

{ 
In this example, we perform the following numerical tests: (i) the spatial convergence rates on a fixed temporal mesh, measured in the $L^2$, $L^\infty$, and energy norms, with and without the penalty term; (ii) the temporal convergence rates on a fixed spatial grid, using three different startup procedures; (iii) the performance of the higher-order LBRFD formulas; and (iv) a comparison between the LBRFD methods and the BDF$p$ methods of the same order. The corresponding results are reported in Tables~\ref{tab:numerical_n6d3_1D}--\ref{tab:numerical_BDF-compare} and Figures~\ref{fig:numerical_BDF12-compare}--\ref{fig:numerical_BDF56-compare}.}

We first examine the spatial discretization errors on a fixed temporal mesh,
measured in the $L^2$, $L^\infty$, and energy norms. Table~\ref{tab:numerical_n6d3_1D} shows that, for given $(n,d)$ and $N$ with the startup scheme \eqref{eq:full_discrete_st}, the $L^2$ errors of the starting values and the { LBRFD} solution in all three norms attain the optimal convergence orders, despite the theoretical prediction of $k-1$ in the $L^2$ norm. Table~\ref{tab:numerical_n5d2} indicates that the
method fails to achieve the optimal order when the penalty parameter
vanishes, whereas the optimal order is recovered once the penalty parameter
is positive. These results suggest that the observed convergence rate
depends on the presence of the penalty term in the scheme.

{ 
Secondly, we assess the convergence behavior of the temporal discretization for the LBRFD method on a fixed spatial grid using three startup procedures---exact initial data, three-stage Radau~IIA, and BRIF \eqref{eq:full_discrete_st}. As reported in Table~\ref{tab:numerical_EX-RK-compare}, with all three startup schemes, the LBRFD method is temporally of orders $1$ through $5$ for $(n,d)=(3,1), (4,1), (5,2), (6,3),$ and $(9,4)$, respectively. Table~\ref{tab:numerical_EX-RK-compare} also shows that the startup strategy has no effect on the convergence order of the LBRFD solution. In Table~\ref{tab:numerical_EX-RK-compare}, for every pair $(n,d)$ and $N=20,24,28,32$, the errors of the approximate solutions obtained by the LBRFD method using either the ``Exact solution'' or the three-stage Radau~IIA startup procedure agree to four decimal places (in scientific notation), since both procedures produce startup values with very small errors.}

{ 
For $(n,d)=(8,5)$, $(10,5)$, and $(12,5)$, the { LBRFD} formula is of order~$6$ for ODEs. In the PDE setting, however, no value of $\eta_n\in[0,1)$ is available to ensure $G$-stability in Lemma~\ref{lem:g_stable}, and hence the sixth-order temporal convergence cannot be justified by the present analysis. Nevertheless, Table~\ref{tab:numerical_6th-compare} reports temporal orders near~$6$ for both startup procedures (exact solution and { BRIF \eqref{eq:full_discrete_st}}), suggesting that the method performs better in practice than the current theory predicts; this is also reflected in the right panel of Figure~\ref{fig:numerical_BDF56-compare}. Note that the startup orders in Table~\ref{tab:numerical_6th-compare} are erratic because the startup errors are already of order $10^{-10}$ and are thus masked by rounding errors; they should not be interpreted as convergence rates.}

Finally, we compare the numerical performance of the BDF$p$ methods against that of the LBRFD scheme, where the starting values are generated via the $L^2$ projection of the exact solution. As shown in Tables~\ref{tab:numerical_BDF-compare} and Figures~\ref{fig:numerical_BDF12-compare}--\ref{fig:numerical_BDF56-compare}, the LBRFD methods with the corresponding parameter pairs $(n,d)$ attain a convergence order of $p$, with $p=d$ when $n-d$ is even and $p=d+1$ when $n-d$ is odd. { For $p\le 4$, suitable choices of $(n,d)$ can yield smaller errors than the corresponding BDF$p$ methods, although this advantage is not observed for $p\ge 5$ or in the two-dimensional case; this should be regarded as a limited numerical observation rather than a general superiority of the LBRFD methods. In addition, although $(n,d)=(6,4)$ violates the condition $\eta\in[0,1)$ required for $G$-stability in Lemma~\ref{lem:g_stable}, the right panel of Figure~\ref{fig:numerical_BDF34-compare} shows that its error attains the optimal order and is the smallest among the fourth-order methods.

}

\begin{table}[!htbp]
\caption{ Spatial errors and convergence rates for Example~\ref{exm5-1} with and without the penalty term, using the fixed temporal mesh $N=160$, and the parameter pair $(n,d)=(5,2)$.}
\label{tab:numerical_n5d2}
\centering
\begin{tabular}{c c c c c c c cccc}
\toprule
&&\multicolumn{4}{c}{penalty coefficient $\beta_0=0$}&\multicolumn{4}{c}{penalty coefficient $\beta_0=k^2$}\\
\cmidrule(lr){3-6} \cmidrule(lr){7-10}
&& \multicolumn{2}{c}{ Startup value \eqref{eq:full_discrete_st}}
& \multicolumn{2}{c}{LBRFDM}&\multicolumn{2}{c}{ Startup value \eqref{eq:full_discrete_st}}
& \multicolumn{2}{c}{LBRFDM}\\
\cmidrule(lr){3-4} \cmidrule(lr){5-6} \cmidrule(lr){7-8}\cmidrule(lr){9-10}
$k$&$M$ & $\|e_h\|_{S,L^2}$ & order & $\|e_h\|_{N,L^2}$ & order & $\|e_h\|_{S,L^2}$ & order & $\|e_h\|_{N,L^2}$ & order \\
\midrule
&16 & 1.2953e-02  && 1.9763e-02&&1.2978e-02&&1.9854e-02&\\
1  &32& 2.8180e-03  &2.20& 4.6166e-03&2.09&2.8242e-03&2.20&4.6212e-03&2.10\\
 &64&  6.7696e-04  &2.05& 1.1336e-03   & 2.02&6.7781e-04&2.06&1.1338e-03&2.03  \\ \hline
&16& 4.0951e-03  && 6.9822e-02&&2.7303e-03&& 2.9543e-03&\\
 2 &32& 1.3264e-03 &1.62&  2.6365e-02&1.40&3.8848e-04&2.81&3.8457e-04&2.94\\
 &64&  4.6369e-04 &1.51&  9.4822e-03&1.47&4.9034e-05&2.98&4.7630e-05&3.01\\ \hline
 &16&1.9506e-05  && 2.7997e-04&&1.4996e-05&&1.5000e-05&\\
3&32&  1.4676e-06 &3.73&  2.4521e-05&3.51&8.8177e-07&4.09&8.5921e-07&4.12\\
 &64&  1.1790e-07  &3.63& 2.1628e-06&3.50&5.4566e-08&4.01&5.5897e-08&3.94\\
\bottomrule
\end{tabular}
\end{table}

\begin{table}[!htbp]
\caption{  Spatial errors and convergence rates for Example~\ref{exm5-1} with the penalty term $\beta_0=10k^2$, the fixed temporal mesh $N=320$, and the parameter pair $(n,d)=(6,3)$.}
\label{tab:numerical_n6d3_1D}
\centering
\begin{tabular}{c c c c c c c ccc}
\toprule
&&\multicolumn{2}{c}{ Starting value \eqref{eq:full_discrete_st}}&\multicolumn{6}{c}{ LBRFDM}\\
\cmidrule(lr){3-4} \cmidrule(lr){5-10}
$k$&$M$&$\|e_h\|_{S,L^2}$ &order & $\|e_h\|_{N,L^\infty}$&order&$\|e_h\|_{N,L^2}$&order&$\|e_h\|_{N,\interleave\cdot\interleave_h}$&order\\
\midrule
&16&1.3032e-02 &  &1.4585e-02 & & 1.9796e-02 &  &2.1243e-01 & \\
1&32&2.7953e-03 & 2.22 &3.4377e-03 & 2.08 & 4.6183e-03 & 2.10 &1.0028e-01 & 1.08 \\
&64&6.6684e-04 & 2.07 &8.4356e-04 & 2.03 & 1.1337e-03 & 2.03 &4.9395e-02 & 1.02\\\hline
&16&2.4404e-03 &  &2.2324e-03 &  & 2.4010e-03 &  &2.7164e-02 & \\
2&32&3.5131e-04 & 2.80 &3.0042e-04 & 2.89 & 3.4530e-04 & 2.80 &7.6154e-03 & 1.83 \\
&64&4.6515e-05 & 2.92 &3.8305e-05 & 2.97 & 4.5657e-05 & 2.92 &1.9941e-03 & 1.93 \\ \hline
&16&1.4354e-05 &  &1.1346e-05 &  & 1.4093e-05 &  &3.3601e-04 &  \\
3&32&8.7855e-07 & 4.03 &7.0587e-07 & 4.01 & 8.6225e-07 & 4.03 &4.1518e-05 & 3.02 \\
&64&5.4602e-08 & 4.01 &4.4069e-08 & 4.00 & 5.3587e-08 & 4.01 &5.1738e-06 & 3.00 \\
\bottomrule
\end{tabular}
\end{table}

\begin{table}[!htbp]
\caption{ Temporal errors and convergence rates for Example~\ref{exm5-1} using three startup procedures and different parameter pairs $(n,d)$, with the penalty term $\beta_0=k^2$, the fixed spatial mesh $M=1000$, and the polynomial degree $k=3$.}
\label{tab:numerical_EX-RK-compare}
\centering
\setlength{\extrarowheight}{0.6mm}
\setlength{\tabcolsep}{1mm}
\begin{tabular}{c  c c c c ccccccccc}  
\toprule
&&\multicolumn{2}{c}{Exact solution}&\multicolumn{4}{c}{3s-Radau IIA}&\multicolumn{4}{c}{ BRIF \eqref{eq:full_discrete_st}}\\
\cmidrule(lr){3-4} \cmidrule(lr){5-8}  \cmidrule(lr){9-12}
&&\multicolumn{2}{c}{LBRFDM}&\multicolumn{2}{c}{Startup value}
& \multicolumn{2}{c}{LBRFDM}&\multicolumn{2}{c}{Startup value}
& \multicolumn{2}{c}{LBRFDM}\\
\cmidrule(lr){3-4} \cmidrule(lr){5-6} \cmidrule(lr){7-8}\cmidrule(lr){9-10}\cmidrule(lr){11-12}
$ (n,d)$&$N$ & $\|e_h\|_{N,L^2}$ & order & $\|e_h\|_{S,L^2}$ & order & $\|e_h\|_{N,L^2}$ & order & $\|e_h\|_{S,L^2}$ & order & $\|e_h\|_{N,L^2}$ & order \\
\midrule
&20& 8.2678e-03&& 6.9029e-10&& 8.2678e-03&& 2.7289e-03&& 8.8615e-03& \\
&24&  7.0539e-03& 0.87& 3.1310e-10& 4.34& 7.0539e-03& 0.87& 1.9610e-03& 1.81& 7.4858e-03& 0.93  \\
$(3,1)$&28& 6.1488e-03& 0.89& 1.6032e-10& 4.34& 6.1488e-03& 0.89& 1.4771e-03& 1.84& 6.4766e-03& 0.94  \\
&32&  5.4494e-03& 0.90& 8.9758e-11& 4.34& 5.4494e-03& 0.90& 1.1526e-03& 1.86& 5.7062e-03& 0.95  \\\hline
&20& 5.4712e-04&& 6.9029e-10&& { 5.4712e-04}&& 1.9886e-04&& 5.3260e-04& \\
&24& 3.8958e-04& 1.86& 3.1310e-10& 4.34& 3.8958e-04& 1.86& 1.1981e-04& 2.78& 3.7982e-04& 1.85  \\
$(4,1)$&28&  2.9160e-04& 1.88& 1.6032e-10& 4.34& 2.9160e-04& 1.88& 7.7675e-05& 2.81& 2.8444e-04& 1.88  \\
&32&  2.2625e-04& 1.90& 8.9758e-11& 4.34& 2.2625e-04& 1.90& 5.3196e-05& 2.83& 2.2107e-04& 1.89  \\\hline
&20&  2.0079e-05&& 6.9029e-10&& { 2.0079e-05}&& 8.0436e-06&& 2.1457e-05& \\
&24& 1.1949e-05& 2.85& 3.1310e-10& 4.34& 1.1949e-05& 2.85& 4.0597e-06& 3.75& 1.2692e-05& 2.88  \\
$(5,2)$&28&  7.6816e-06& 2.87& 1.6032e-10& 4.34& 7.6816e-06& 2.87& 2.2644e-06& 3.79& 8.1101e-06& 2.91  \\
&32&  5.2254e-06& 2.89& 8.9758e-11& 4.34& 5.2254e-06& 2.89& 1.3607e-06& 3.81& 5.4909e-06& 2.92  \\\hline
&20& 7.8273e-07&& 6.9029e-10&& { 7.8271e-07}&& 3.4727e-07&& 7.6072e-07& \\
&24&  3.8991e-07& 3.82& 3.1310e-10& 4.34& 3.8990e-07& 3.82& 1.4673e-07& 4.73& 3.7908e-07& 3.82  \\
$(6,3)$&28& 2.1553e-07& 3.85& 1.6032e-10& 4.34& 2.1553e-07& 3.85& 7.0373e-08& 4.77& 2.0953e-07& 3.85  \\
&32& 1.2859e-07& 3.87& 8.9758e-11& 4.34& 1.2859e-07& 3.87& 3.7093e-08& 4.80& 1.2514e-07& 3.86  \\\hline
&20&  4.4583e-08&& 6.9029e-10&& { 4.4616e-08}&& 2.5116e-08&& 4.8105e-08& \\
&24&  1.8975e-08& 4.69& 3.1310e-10& 4.34& 1.8986e-08& 4.69& 8.9530e-09& 5.66& 2.0151e-08& 4.77  \\
$(9,4)$&28& 9.0804e-09& 4.78& 1.6032e-10& 4.34& 9.0847e-09& 4.78& 3.7157e-09& 5.70& 9.6212e-09& 4.80  \\
&32&  4.7748e-09& 4.81& 8.9758e-11& 4.34& 4.7770e-09& 4.81& 1.7304e-09& 5.72& 5.0364e-09& 4.85  \\
\bottomrule
\end{tabular}
\vspace{0.5em}
{\footnotesize\raggedright Note: For every pair $(n,d)$ listed in this table, the ``Exact solution'' startup procedure yields the same errors $\|e_h\|_{S,L^2}=5.4811\times10^{-13}$ for $N=20,24,28,32$.\par}
\end{table}
\begin{table}[!htbp]
\caption{ Temporal errors and convergence rates for Example~\ref{exm5-1} using two startup procedures and different parameter pairs $(n,d)$, with the penalty term $\beta_0=k^2$, the fixed spatial mesh $M=1000$, the polynomial degree $k=3$ and 6-th method.}
\label{tab:numerical_6th-compare}
\centering
\begin{tabular}{c c c c c c c cccc}  
\toprule
&&\multicolumn{4}{c}{Exact solution}&\multicolumn{4}{c}{ BRIF \eqref{eq:full_discrete_st}}\\
\cmidrule(lr){3-6} \cmidrule(lr){7-10}
&& \multicolumn{2}{c}{Startup value}
& \multicolumn{2}{c}{LBRFDM}&\multicolumn{2}{c}{Startup value}
& \multicolumn{2}{c}{LBRFDM}\\
\cmidrule(lr){3-4} \cmidrule(lr){5-6} \cmidrule(lr){7-8}\cmidrule(lr){9-10}
$ (n,d)$&$N$ & $\|e_h\|_{S,L^2}$ & order & $\|e_h\|_{N,L^2}$ & order & $\|e_h\|_{S,L^2}$ & order & $\|e_h\|_{N,L^2}$ & order \\
\midrule
&20& 5.4810e-13&& 1.3247e-09&& 7.0706e-10&& 1.2914e-09& \\
&24& 5.4810e-13& 0.00& 4.6761e-10& 5.71& 4.1112e-10& 2.97& 4.5558e-10& 5.71  \\
$(8,5)$&28& 5.4810e-13& 0.00& 1.8949e-10& 5.86& 4.2477e-10& -0.21& 1.8425e-10& 5.87  \\
&32& 5.4810e-13& 0.00& 8.2562e-11& 6.22& 2.7958e-10& 3.13& 8.0101e-11& 6.24  \\\hline
&20& 5.4810e-13&& 1.8184e-09&& 1.1294e-09&& 1.7803e-09& \\
&24& 5.4810e-13& 0.00& 6.6091e-10& 5.55& 4.7039e-10& 4.80& 6.4378e-10& 5.58  \\
$(10,5)$&28& 5.4810e-13& 0.00& 2.7039e-10& 5.80& 4.3644e-10& 0.49& 2.6292e-10& 5.81  \\
&32& 5.4810e-13& 0.00& 1.2529e-10& 5.76& 4.4356e-10& -0.12& 1.2136e-10& 5.79  \\\hline
&20& 5.4810e-13&& 2.1664e-09&& 1.5700e-09&& 2.1019e-09& \\
&24& 5.4810e-13& 0.00& 8.2313e-10& 5.31& 4.7437e-10& 6.56& 8.0163e-10& 5.29  \\
$(12,5)$&28& 5.4810e-13& 0.00& 3.4657e-10& 5.61& 4.3784e-10& 0.52& 3.3894e-10& 5.58  \\
&32& 5.4810e-13& 0.00& 1.6141e-10& 5.72& 4.4146e-10& -0.06& 1.5791e-10& 5.72  \\
\bottomrule
\end{tabular}
\end{table}
\begin{table}[!htbp]
\caption{ Comparison of temporal errors and convergence rates between the BDF methods and the LBRFD methods of the same order for Example~\ref{exm5-1}, using ``Exact solution'' startup procedure, with the penalty term $\beta_0=k^2$, the fixed spatial mesh $M=1000$, and the polynomial degree $k=3$.}
\label{tab:numerical_BDF-compare}
\centering
\begin{tabular}{c c c c c c c cccc}  
\toprule
&\multicolumn{3}{c}{BDF}&\multicolumn{3}{c}{LBRFDM}\\
\cmidrule(lr){2-4} \cmidrule(lr){5-7}
$N$ & BDF & $\|e_h\|_{N,L^2}$ & order & $ (n,d)$ & $\|e_h\|_{N,L^2}$ & order \\
\midrule
20& & 2.9141e-04& & & 5.4712e-04& \\
24&   & 2.0529e-04& 1.92&  & 3.8958e-04& 1.86  \\
28& BDF2& 1.5247e-04& 1.93& $(4,1)$ & 2.9160e-04& 1.88  \\
32& & 1.1764e-04& 1.94&  & 2.2625e-04& 1.90  \\\hline
20&  & 1.0643e-05& & & 2.0079e-05& \\
24&  & 6.2751e-06& 2.90&  & 1.1949e-05& 2.85  \\
28& BDF3& 4.0064e-06& 2.91&  $(5,2)$& 7.6816e-06& 2.87  \\
32&  & 2.7112e-06& 2.92&  & 5.2254e-06& 2.89  \\\hline
20&  & 4.1435e-07& & & 7.8273e-07& \\
24&  & 2.0453e-07& 3.87&  & 3.8991e-07& 3.82  \\
28& BDF4& 1.1230e-07& 3.89&  $(6,3)$& 2.1553e-07& 3.85  \\
32& & 6.6658e-08& 3.91&  & 1.2859e-07& 3.87  \\\hline
20&  & 1.6835e-08& & & 4.4583e-08& \\
24&  & 6.9509e-09& 4.85&  & 1.8975e-08& 4.69  \\
28& BDF5& 3.2812e-09& 4.87&  $(9,4)$& 9.0804e-09& 4.78  \\
32&  & 1.7107e-09& 4.88&  & 4.7748e-09& 4.81  \\
\bottomrule
\end{tabular}
\end{table}

\begin{figure}[!htbp]
\centering
\vspace{-0.2cm}
\subfigure[BDF1]{
\label{fig:numerical_BDF1-compare}
\includegraphics[scale=0.30]{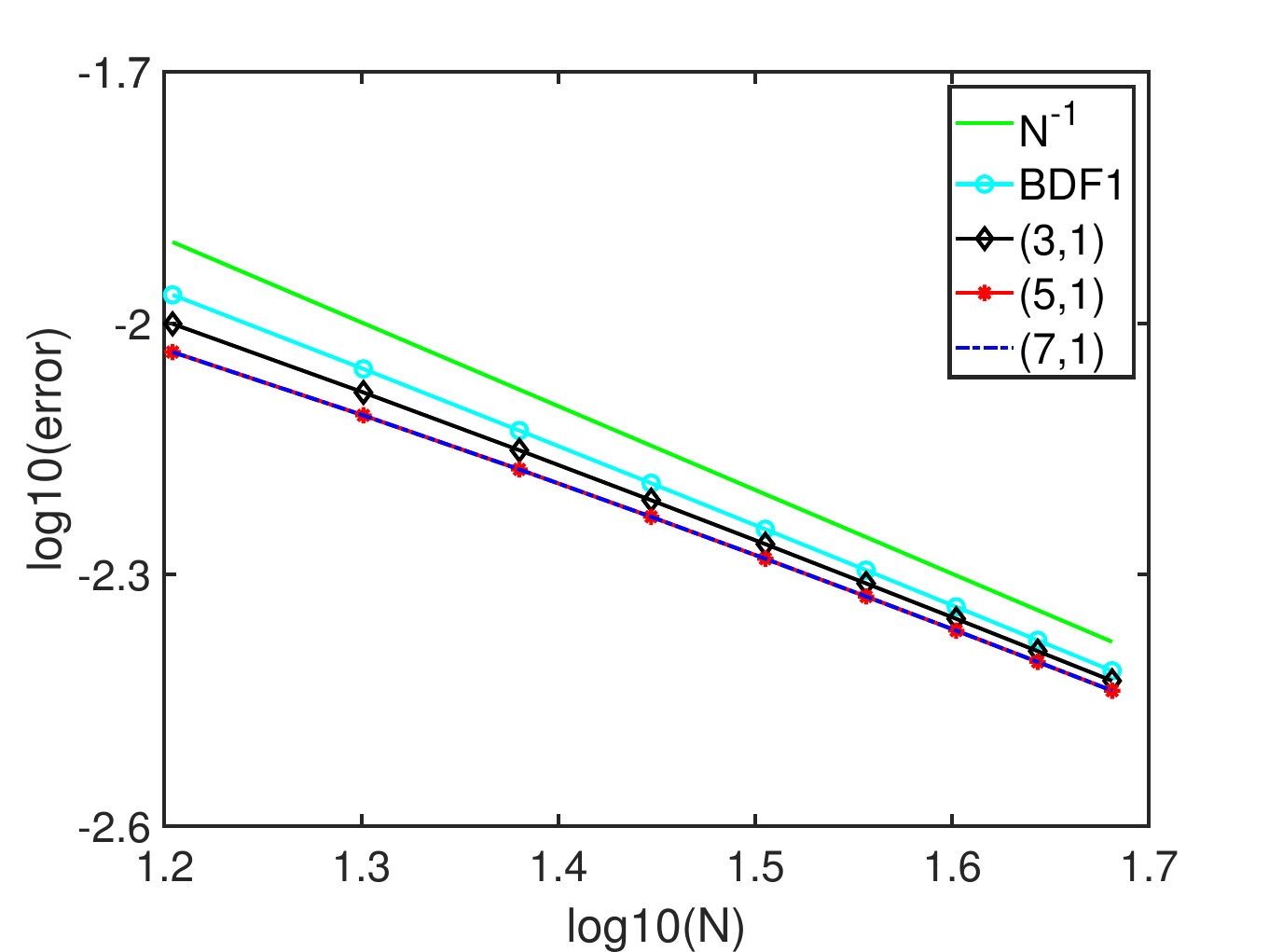}}
\subfigure[BDF2]{
\label{fig:numerical_BDF2-compare}
\includegraphics[scale=0.30]{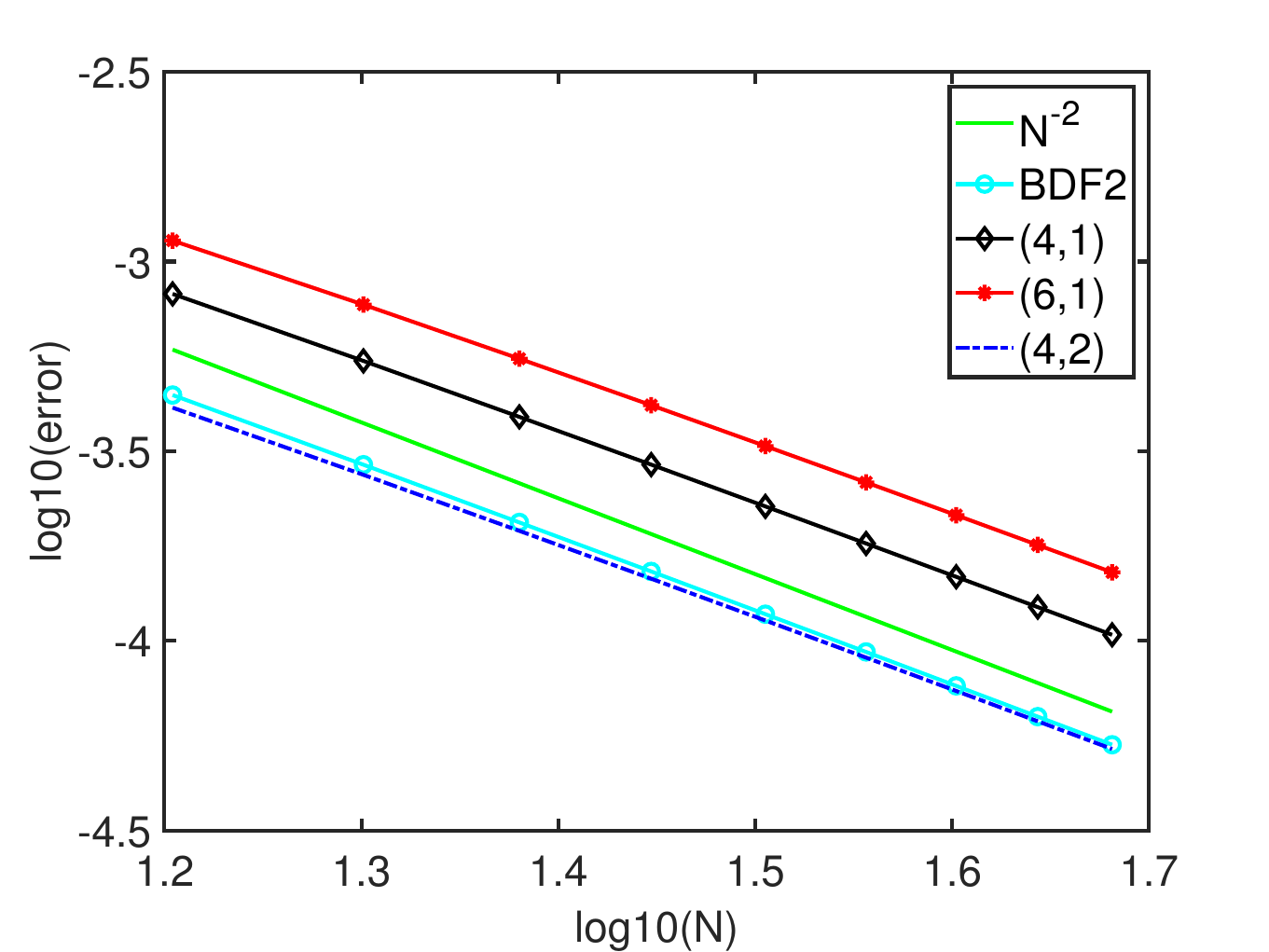}}
\vspace{-0.2cm}
\caption{Numerical results of Example \ref{exm5-1} with $M=1000$, $k=3$, $\beta_0=k^2$.}
\label{fig:numerical_BDF12-compare}
\vspace{-0.2cm}
\end{figure}
\begin{figure}[!htbp]
\centering
\vspace{-0.2cm}
\subfigure[BDF3]{
\label{fig:numerical_BDF3-compare}
\includegraphics[scale=0.30]{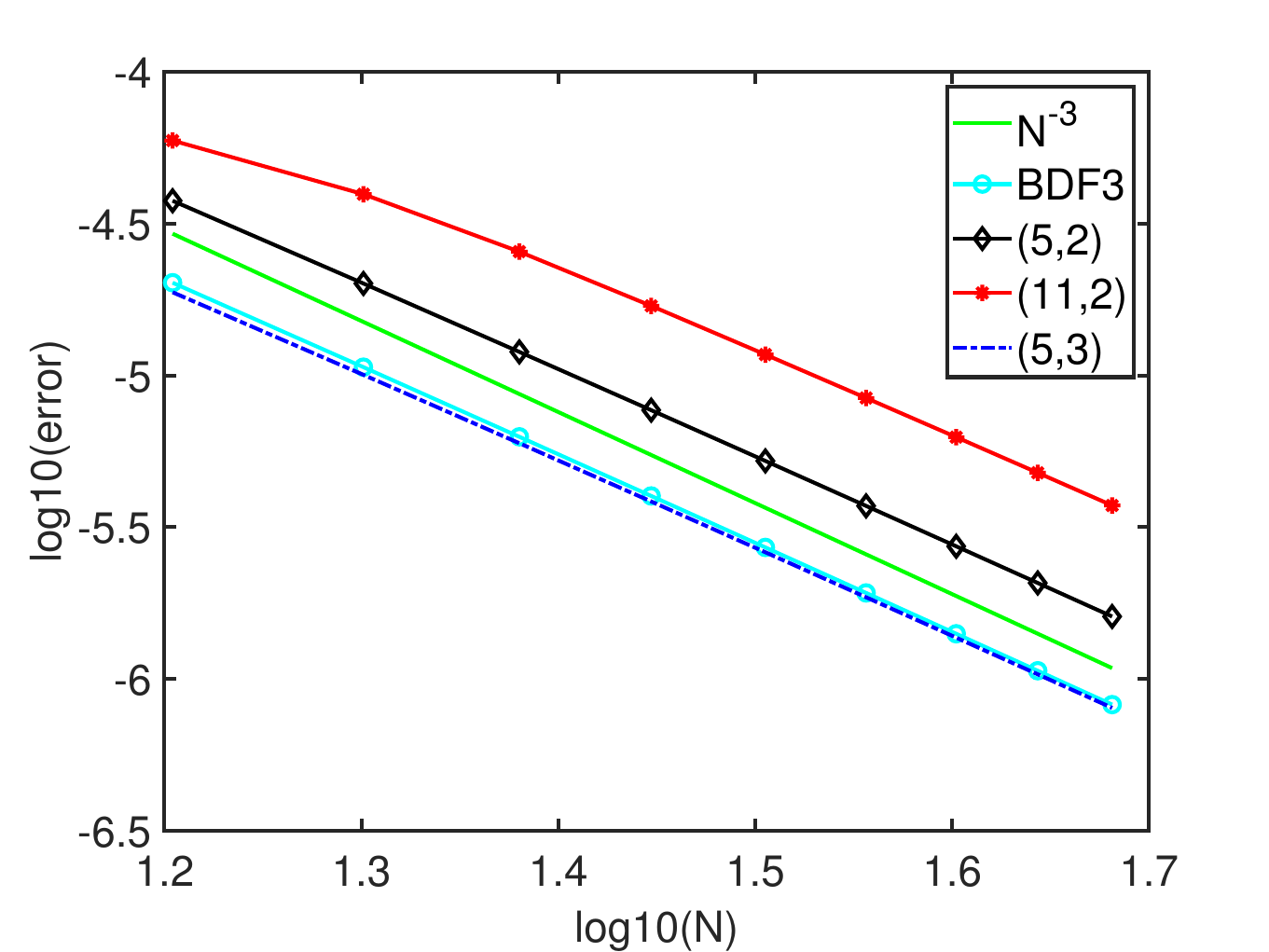}}
\subfigure[BDF4]{
\label{fig:numerical_BDF4-compare}
\includegraphics[scale=0.30]{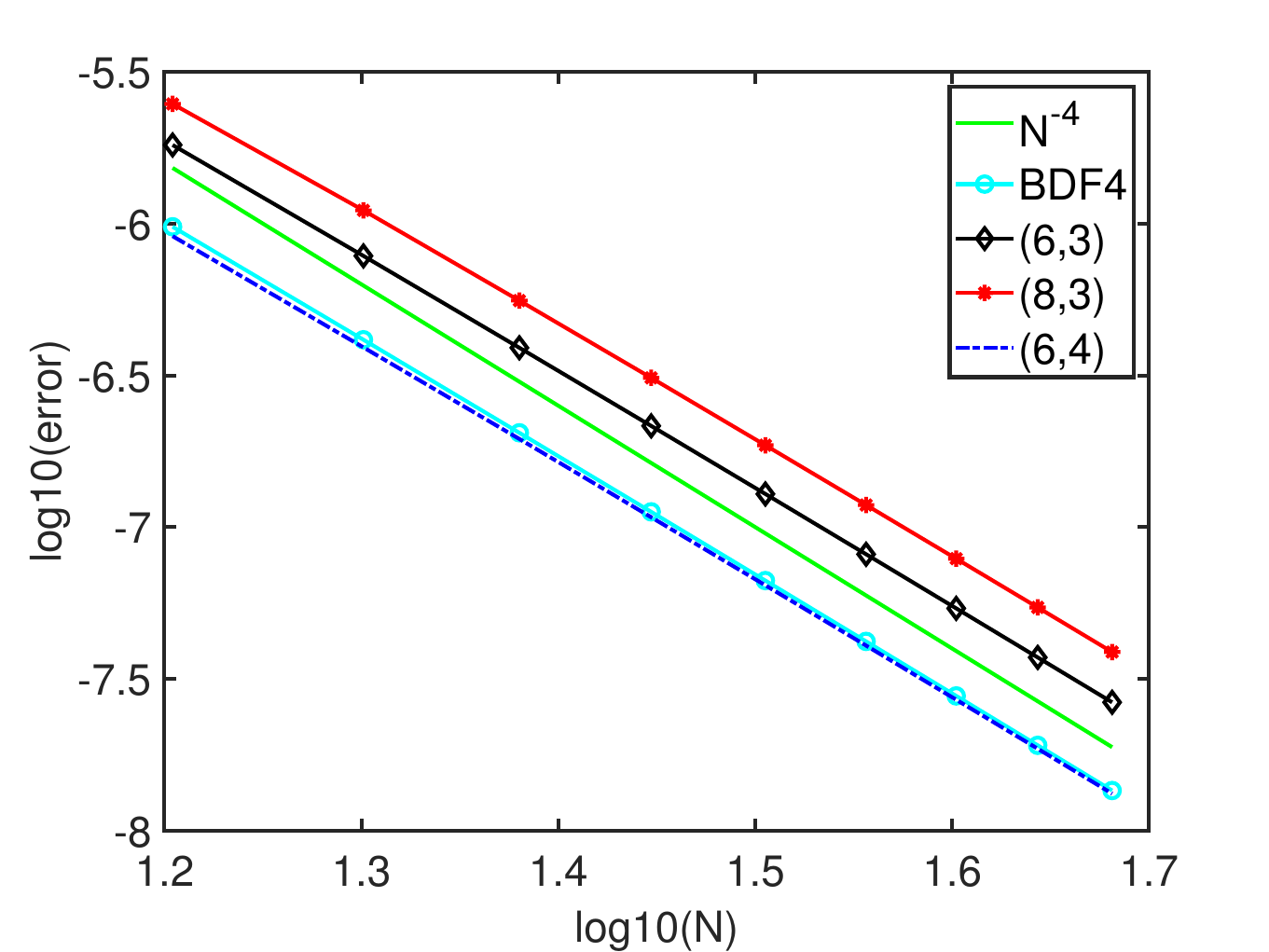}}
\vspace{-0.2cm}
\caption{Numerical results of Example \ref{exm5-1} with $M=1000$, $k=3$, $\beta_0=k^2$.}
\label{fig:numerical_BDF34-compare}
\vspace{-0.2cm}
\end{figure}
\begin{figure}[!htbp]
\centering
\vspace{-0.2cm}
\subfigure[BDF5]{
\label{fig:numerical_BDF5-compare}
\includegraphics[scale=0.30]{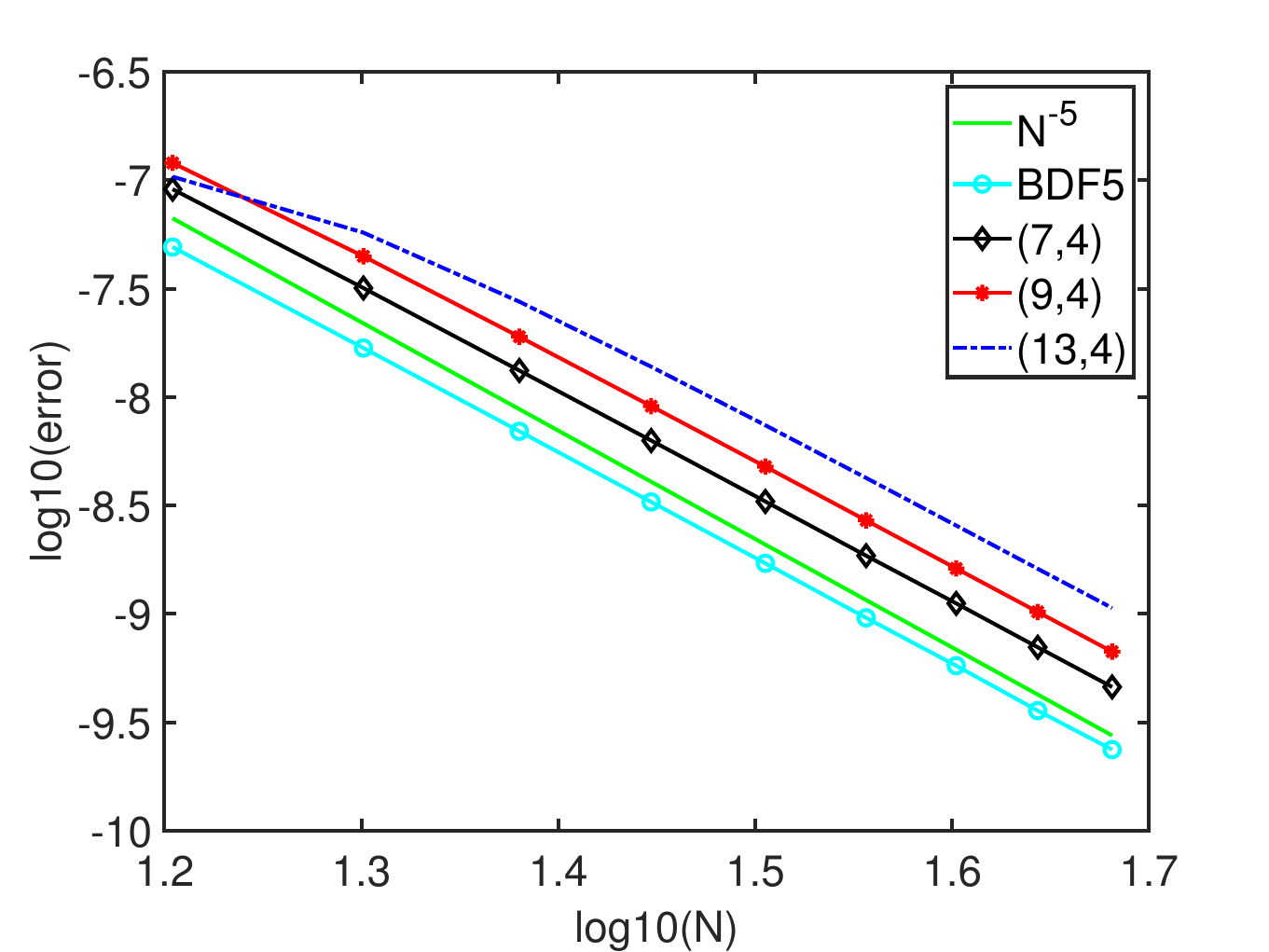}}
\subfigure[BDF6]{
\label{fig:numerical_BDF6-compare}
\includegraphics[scale=0.30]{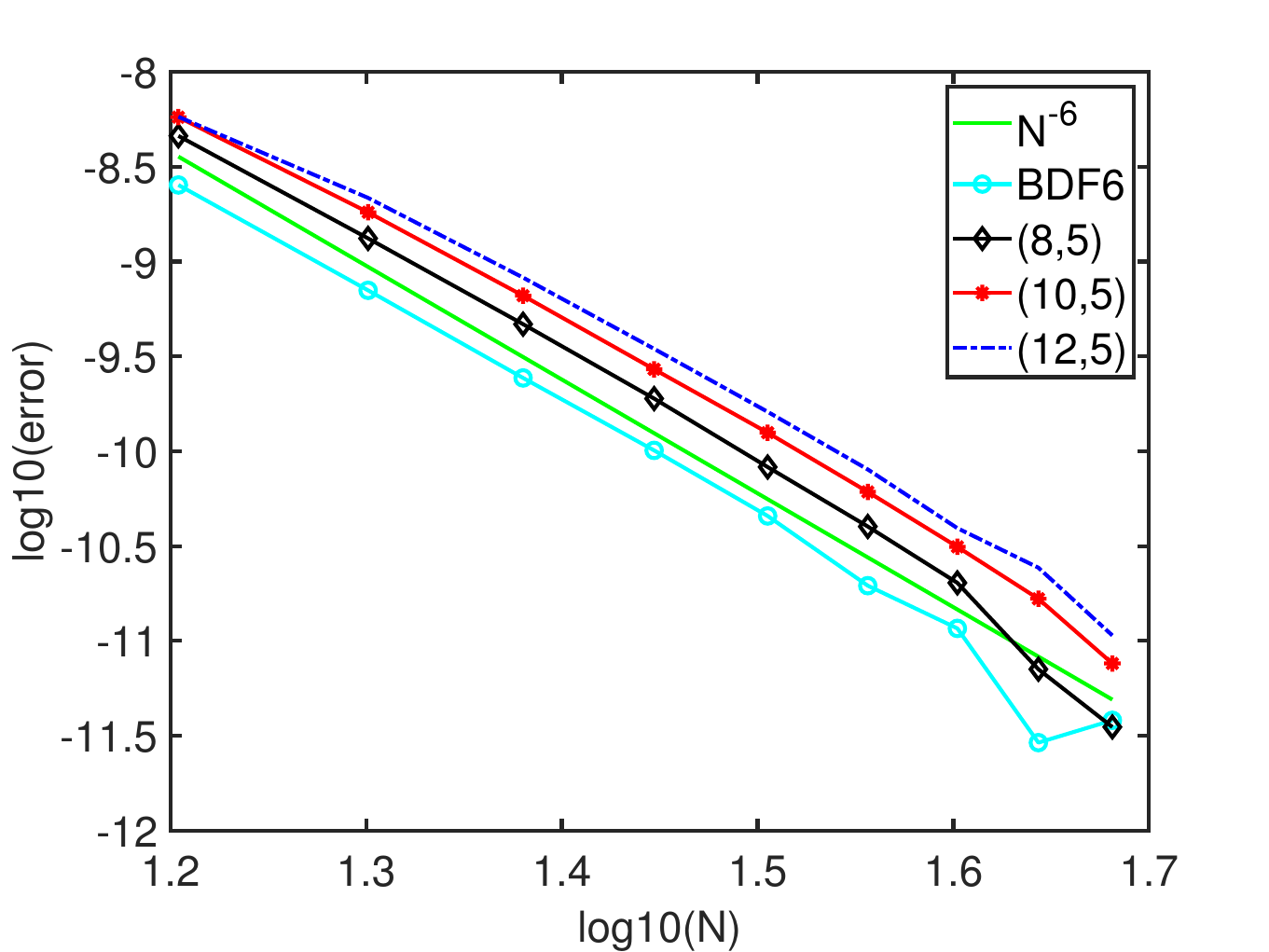}}
\vspace{-0.2cm}
\caption{Numerical results of Example \ref{exm5-1} with $M=1000$, $k=3$, $\beta_0=k^2$.}
\label{fig:numerical_BDF56-compare}
\vspace{-0.2cm}
\end{figure}

\begin{example}\label{exm5-2}
Consider the following two-dimensional fourth-order problem
$$u_t+\Delta^2u=0,\quad u(x,y,0)=\sin(x+y),\quad (x,y)\in \Omega=[0,b]\times [0,b], t\in (0,0.1],$$
subject to Dirichlet boundary conditions
$$u|_{\p\Omega}=g_1(x,y,t),\quad \left.\frac{\p u}{\p\mathbf{n}}\right|_{\p\Omega}=g_2(x,y,t),$$
 chosen so that the exact solution is given by
\begin{equation}
u(x,y,t) = e^{-4t}\sin(x+y).
\end{equation}
\end{example}

In this example, the domain $[0,b]\times[0,b]$ is discretized using a uniform triangular mesh. To this end, the interval $[0,b]$ is first divided into $M$ equal subintervals in each spatial direction, yielding a uniform rectangular partition of the square. Each rectangular cell is then divided into two triangles by connecting one of its diagonals.

As in the one-dimensional case, we perform the following numerical tests: (i) the spatial convergence rates on a fixed temporal mesh, measured in the $L^2$, $L^\infty$, and energy norms; (ii) the temporal convergence rates on a fixed spatial grid, using three different initialization strategies; and (iii) a comparison between the LBRFD methods and the BDF$p$ methods of the same order. The corresponding results are reported in Tables~\ref{tab:exm52_numerical_n6d3_2D}--\ref{tab:exm52_numerical_EX-RK-compare-t} and Figures~\ref{fig:numerical_2d-BDF12-compare}--\ref{fig:numerical_2d-BDF34-compare}.

Fist, we examine the spatial discretization errors on a fixed temporal mesh, measured in the $L^2$, $L^\infty$, and energy norms $\interleave\cdot\interleave_h$. Table~\ref{tab:exm52_numerical_n6d3_2D} shows that, for the parameters $(n,d)=(6,3)$ and $N=320$, the convergence orders in the $L^2$, $L^\infty$, and $\interleave\cdot\interleave_h$ norms are all optimal. { Table~\ref{tab:exm52_numerical_EX-RK-compare-h} presents the results obtained using three different approaches for computing the starting values. The numerical results indicate that these three startup procedures have virtually no influence on  the corresponding convergence orders of the LBRFD approximate solutions.}

Next, we examine the temporal discretization errors. Table~\ref{tab:exm52_numerical_EX-RK-compare-t} presents the results obtained with the three { startup procedures}. For given parameters $(n,d)$, the temporal convergence order is $d+1$ when $n-d$ is odd, and $d$ when $n-d$ is even.

Finally, we compare the numerical performance of the LBRFDM scheme against that of the BDF$p$ methods. For the LBRFDM, starting values are generated using three different startup strategies, while for the BDF$p$ methods, they are obtained via the $L^2$ projection of the exact solution. As in the one-dimensional case, Figures~\ref{fig:numerical_2d-BDF12-compare}--\ref{fig:numerical_2d-BDF34-compare} demonstrate that the LBRFDM methods with the corresponding parameter pairs $(n,d)$ achieve the same temporal convergence orders as the BDF$p$ methods of the same order.

\begin{table}[!htbp]
\caption{ Spatial errors and convergence rates for Example~\ref{exm5-2} with $(n,d)=(6,3)$, $\beta_0=k^2$, $b=2\pi$ and the fixed temporal mesh $N=320$.}
\label{tab:exm52_numerical_n6d3_2D}
\centering
\begin{tabular}{c c c c c c c ccc}
\toprule
&&\multicolumn{2}{c}{ Starting value \eqref{eq:full_discrete_st}}&\multicolumn{6}{c}{ LBRFDM}\\
\cmidrule(lr){3-4} \cmidrule(lr){5-10}
$k$&$M$&$\|e_h\|_{S,L^2}$ &order & $\|e_h\|_{N,L^\infty}$&order&$\|e_h\|_{N,L^2}$&order&$\|e_h\|_{N,\interleave\cdot\interleave_h}$&order\\
\midrule
&16&1.9742e-02 & &1.1516e-02 &  & 1.9913e-02 &  &5.6717e-01 &  \\
1&32&4.9424e-03 & 2.00 &2.8595e-03 & 2.01 & 5.0852e-03 & 1.97 &2.8448e-01 & 1.00 \\
&64&1.2360e-03 & 2.00 &7.1119e-04 & 2.01 & 1.2835e-03 & 1.99 &1.4241e-01 & 1.00 \\\hline
&16&6.3967e-04 &  &4.0292e-04 &  & 9.7593e-04 & &2.9872e-02 &  \\
2&32&8.0047e-05 & 3.00 &5.7733e-05 & 2.80 & 1.2257e-04 & 2.99 &7.5027e-03 & 1.99 \\
&64&1.0009e-05 & 3.00 &8.1704e-06 & 2.82 & 1.5358e-05 & 3.00 &1.8792e-03 & 2.00 \\\hline
&16&7.7268e-06 &  &1.0723e-05 &  & 2.3327e-05 &  &9.4948e-04 &  \\
3&32&4.8334e-07 & 4.00 &6.8007e-07 & 3.98 & 1.4609e-06 & 4.00 &1.1894e-04 & 3.00 \\
&64&3.0215e-08 & 4.00 &4.2685e-08 & 3.99 & 9.1337e-08 & 4.00 &1.4879e-05 & 3.00 \\
\bottomrule
\end{tabular}
\end{table}

\begin{table}[!htbp]
\caption{ Spatial errors and convergence rates for Example~\ref{exm5-2} using three startup procedures, with  $\beta_0=4k^2$, $(n,d)=(5,2)$, $b=2\pi$ and the fixed temporal mesh $N=80$.}
\label{tab:exm52_numerical_EX-RK-compare-h}
\centering
\setlength{\extrarowheight}{0.6mm}
\setlength{\tabcolsep}{1mm}
\begin{tabular}{c  c c c c ccccccccc}  
\toprule
&&\multicolumn{2}{c}{Exact solution}&\multicolumn{4}{c}{3s-Radau IIA}&\multicolumn{4}{c}{ BRIF \eqref{eq:full_discrete_st}}\\
\cmidrule(lr){3-4} \cmidrule(lr){5-8}  \cmidrule(lr){9-12}
&&\multicolumn{2}{c}{LBRFDM}&\multicolumn{2}{c}{Startup value}
& \multicolumn{2}{c}{LBRFDM}&\multicolumn{2}{c}{Startup value}
& \multicolumn{2}{c}{LBRFDM}\\
\cmidrule(lr){3-4} \cmidrule(lr){5-6} \cmidrule(lr){7-8}\cmidrule(lr){9-10}\cmidrule(lr){11-12}
$k$&$M$ & $\|e_h\|_{N,L^2}$ & order & $\|e_h\|_{S,L^2}$ & order & $\|e_h\|_{N,L^2}$ & order & $\|e_h\|_{S,L^2}$ & order & $\|e_h\|_{N,L^2}$ & order \\
\midrule
 &8&1.6413e-01 &  &1.6443e-01 & &1.6867e-01 &  & 1.6422e-01 &  &1.6868e-01 &  \\
1&16&4.5316e-02 & 1.86 &4.6385e-02 & 1.83 &4.6322e-02 & 1.86 & 4.6430e-02 & 1.82 &4.6232e-02 & 1.87 \\
&32&1.1270e-02 & 2.01 &1.1568e-02 & 2.00 &1.1555e-02 & 2.00 & 1.1574e-02 & 2.00 &1.1545e-02 & 2.00 \\ \hline
&8&1.4750e-02 &  &1.4940e-02 &  &1.4753e-02 & & 1.4943e-02 &  &1.4695e-02 &  \\
2&16&1.9258e-03 & 2.94 &1.9512e-03 & 2.94 &1.9144e-03 & 2.95 & 1.9556e-03 & 2.93 &1.9139e-03 & 2.94 \\
&32&2.4106e-04 & 3.00 &2.4699e-04 & 2.98 &2.4143e-04 & 2.99 & 2.4746e-04 & 2.98 &2.4142e-04 & 2.99 \\\hline
&8&5.5765e-04 &  &5.5663e-04 &  &5.4701e-04 &  & 5.5990e-04 &  &5.4583e-04 &  \\
3&16&3.5843e-05 & 3.96 &3.6563e-05 & 3.93 &3.5843e-05 & 3.93 & 3.6667e-05 & 3.93 &3.5835e-05 & 3.93 \\
&32&2.3065e-06 & 3.96 &2.3637e-06 & 3.95 &2.3111e-06 & 3.96 & 2.3689e-06 & 3.95 &2.3108e-06 & 3.95 \\
\bottomrule
\end{tabular}
\end{table}

\begin{table}[!htbp]
\caption{ Temporal errors and convergence rates for Example~\ref{exm5-2} using three startup procedures, with  $\beta_0=4k^2$ and the fixed spatial mesh $M=32$.}
\label{tab:exm52_numerical_EX-RK-compare-t}
\centering
\begin{tabular}{cc  c c c c cccc}  
\toprule
&&&&\multicolumn{2}{c}{Exact solution}&\multicolumn{2}{c}{3s-Radau IIA}&\multicolumn{2}{c}{ BRIF \eqref{eq:full_discrete_st}}\\
\cmidrule(lr){5-6} \cmidrule(lr){7-8}  \cmidrule(lr){9-10}
$(n,d)$&$b$&$k$&$N$ & $\|e_h\|_{N,L^2}$ & order  & $\|e_h\|_{N,L^2}$ & order &  $\|e_h\|_{N,L^2}$ & order \\
\midrule
  &&&68&2.1207e-03 &  &2.1204e-03 &  &2.1661e-03 &\\
&&&70&2.0626e-03 & 0.96 &2.0624e-03 & 0.96 &2.1055e-03 &0.98\\
(3,1)&$2\pi$&2&72&2.0076e-03 & 0.96 &2.0074e-03 & 0.96 &2.0481e-03 &0.98\\
&&&74&1.9555e-03 & 0.96 &1.9553e-03 & 0.96 &1.9939e-03 &0.98\\\hline
&&&68&1.6519e-05 &  &1.6519e-05 & &1.6261e-05 &\\
&&&70&1.5621e-05 & 1.93 &1.5621e-05 & 1.93 &1.5384e-05 &1.91\\
(4,1)&$2\pi$&3&72&1.4795e-05 & 1.93 &1.4795e-05 & 1.93 &1.4578e-05 &1.91\\
&&&74&1.4034e-05 & 1.93 &1.4034e-05 & 1.93 &1.3834e-05 &1.91\\\hline
&&&20&6.8394e-09 & 2.88 &6.8511e-09 & 2.87 &7.2899e-09 &2.93\\
&&&25&3.5624e-09 & 2.92 &3.5637e-09 & 2.93 &3.7360e-09 &3.00\\
(5,2)&1.0&3&30&2.0823e-09 & 2.94 &2.0818e-09 & 2.95 &2.1449e-09 &3.04\\
&&&35&1.3251e-09 & 2.93 &1.3250e-09 & 2.93 &1.3396e-09 &3.05\\\hline
&&&18&1.1315e-10 & &1.1851e-10 &  &1.1995e-10 &\\
&&&20&7.8005e-11 & 3.90 &8.1644e-11 & 3.91 &8.1990e-11 &3.99\\
(6,3)&1.0&4&24&5.5446e-11 & 3.92 &5.7956e-11 & 3.94 &5.7761e-11 &4.03\\
&&&26&4.0440e-11 & 3.94 &4.2196e-11 & 3.96 &4.1741e-11 &4.06\\
\bottomrule
\end{tabular}
\end{table}

\begin{figure}[!htbp]
\centering
\vspace{-0.2cm}
\subfigure[$(n,d)=(3,1)$]{
\label{fig:numerical_2d-BDF1-compare}
\includegraphics[scale=0.30]{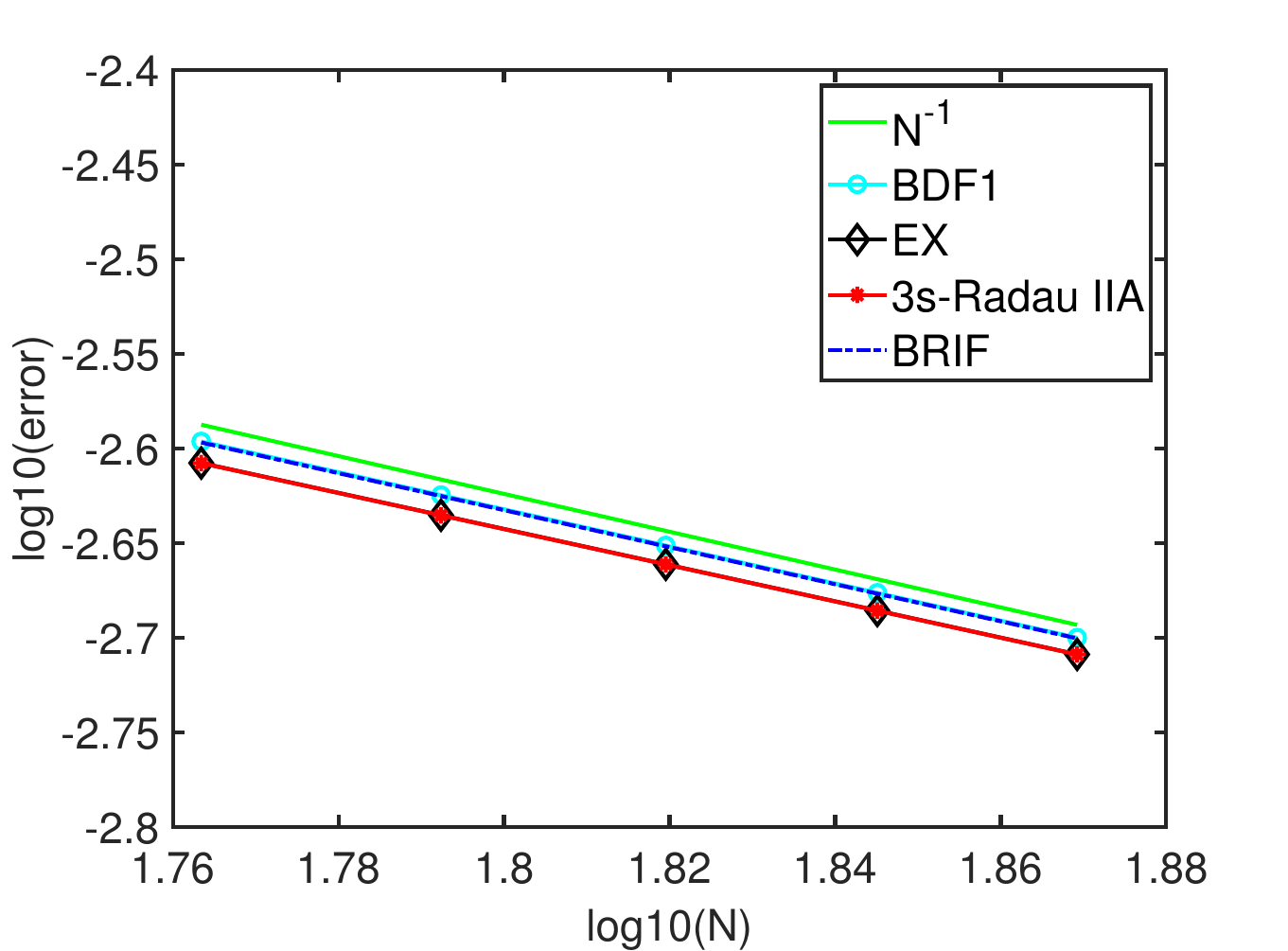}}
\subfigure[$(n,d)=(4,1)$]{
\label{fig:numerical_2d-BDF2-compare}
\includegraphics[scale=0.30]{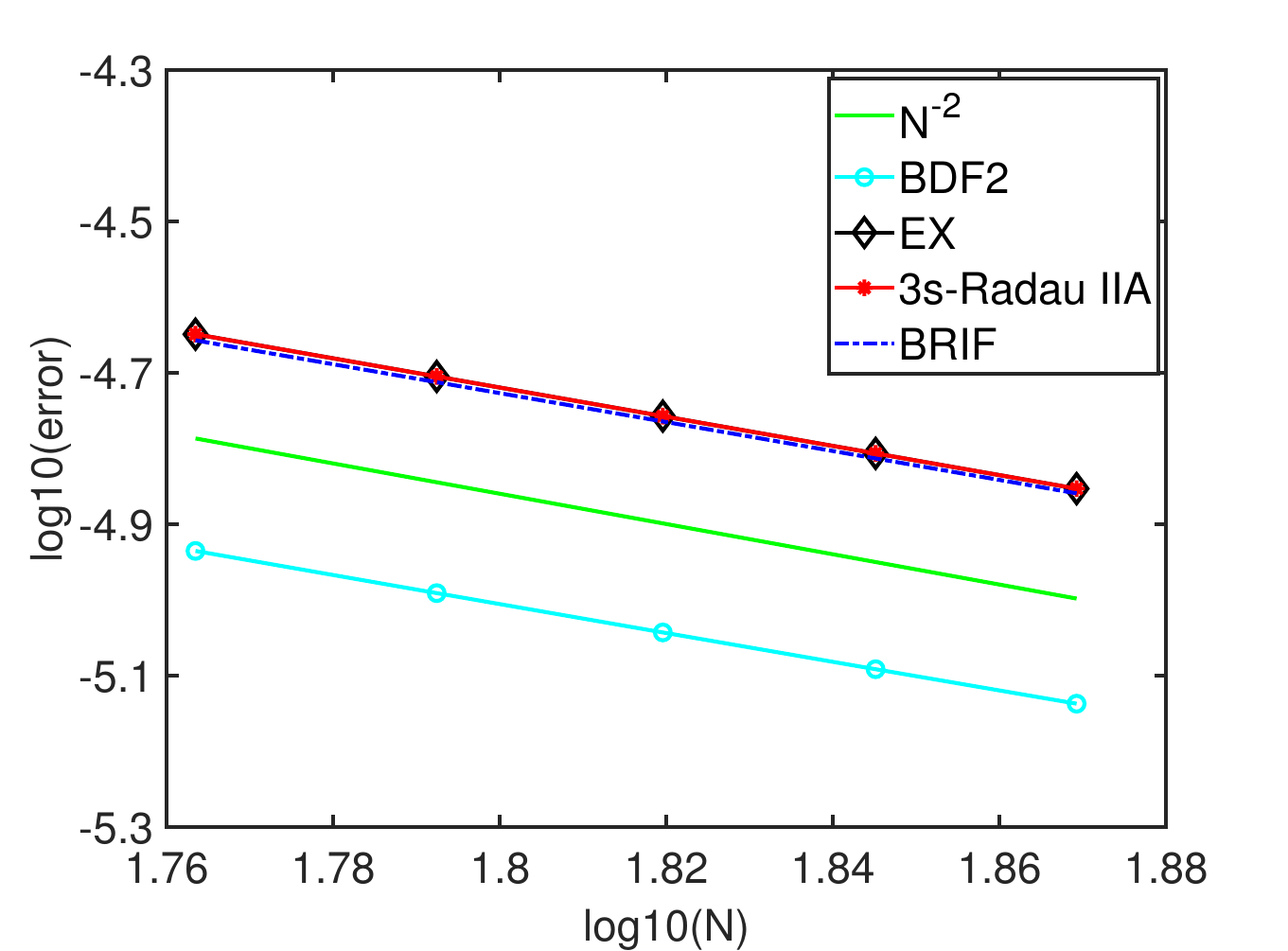}}
\vspace{-0.2cm}
\caption{Numerical results of Example \ref{exm5-2} with $M=32$, $\beta_0=4k^2$, $k=2$ (left), $k=3$ (right).}
\label{fig:numerical_2d-BDF12-compare}
\vspace{-0.2cm}
\end{figure}
\begin{figure}[!htbp]
\centering
\vspace{-0.2cm}
\subfigure[$(n,d)=(5,2)$]{
\label{fig:numerical_2d-BDF3-compare}
\includegraphics[scale=0.30]{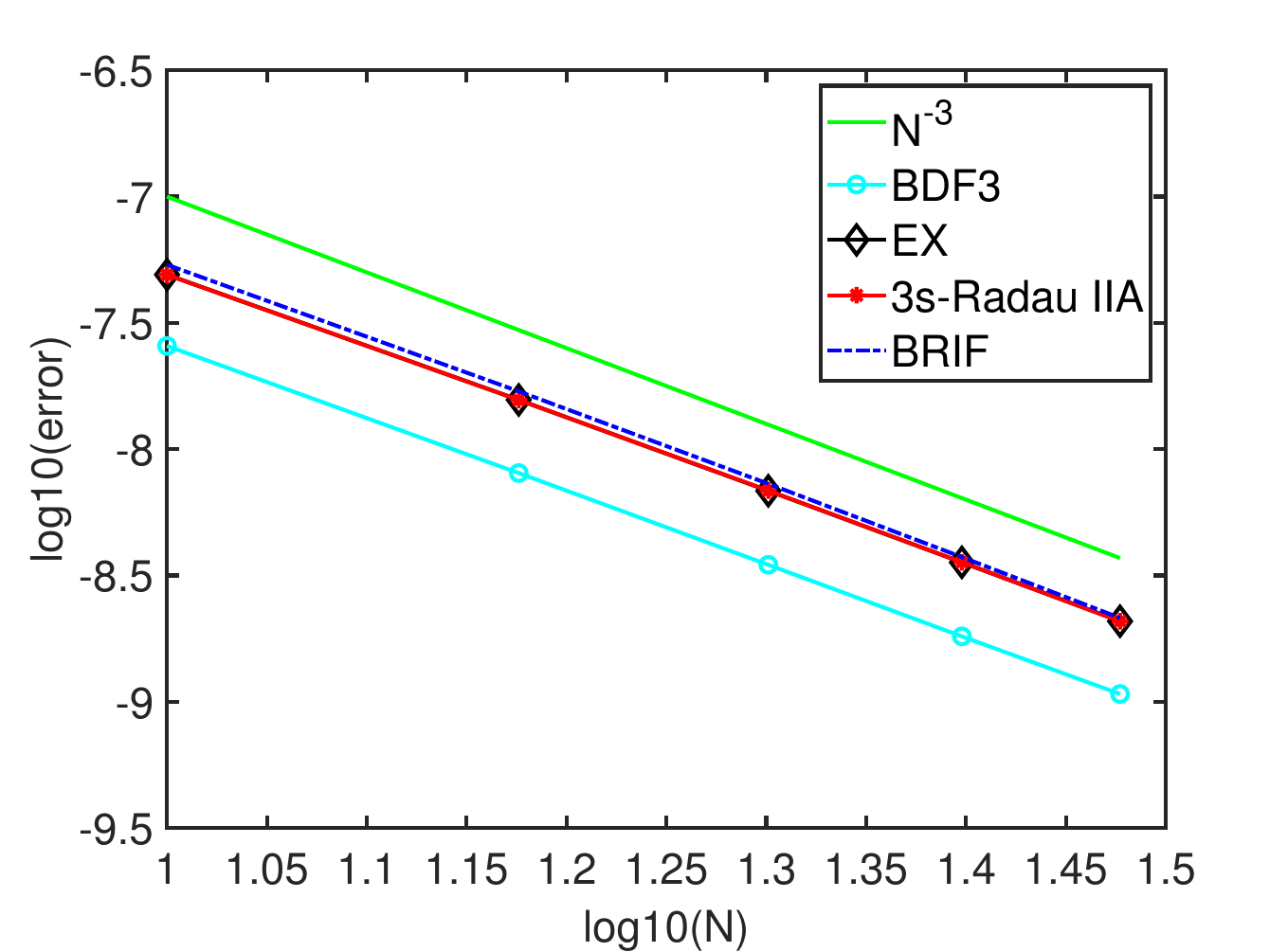}}
\subfigure[$(n,d)=(6,3)$]{
\label{fig:numerical_2d-BDF4-compare}
\includegraphics[scale=0.30]{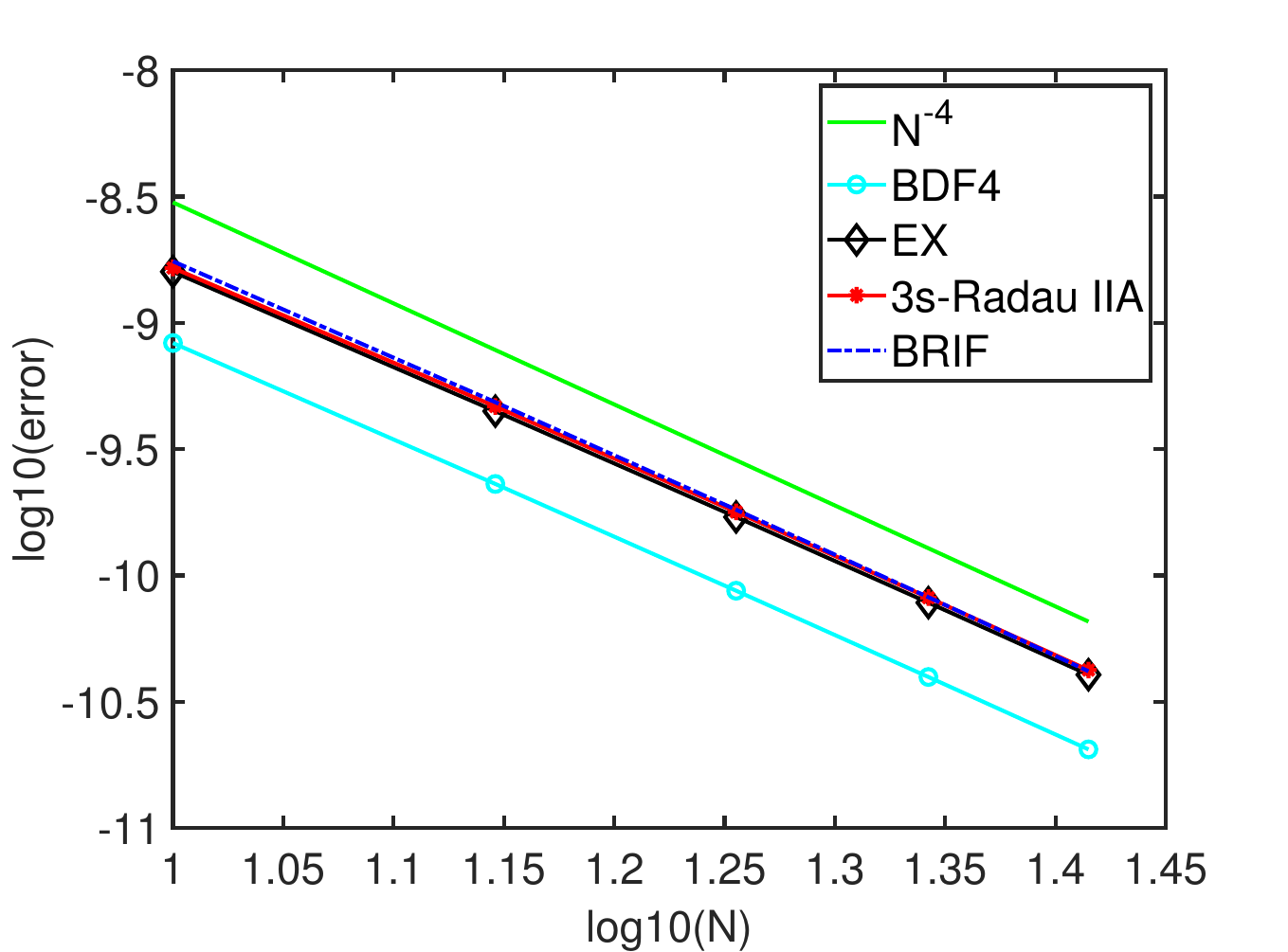}}
\vspace{-0.2cm}
\caption{Numerical results of Example \ref{exm5-2} with $M=32$, $\beta_0=4k^2$, $k=3$ (left), $k=4$ (right).}
\label{fig:numerical_2d-BDF34-compare}
\vspace{-0.2cm}
\end{figure}

\section{Conclusions}

In this paper, we have developed and analyzed a fully discrete method for linear fourth-order parabolic equations, combining an implicit LBRFD multistep scheme in time with a mixed interior penalty discontinuous Galerkin (IPDG) method in space. The temporal discretization employs equispaced linear barycentric rational interpolants and incorporates a startup procedure, while the original problem is reformulated through an auxiliary variable to facilitate the spatial discretization.

For certain parameter pairs $(n,d)$, the LBRFD method is shown to be $A(\alpha)$-stable with a wider stability angle than the corresponding BDF method of the same order. Stability and a priori error estimates are established via a $G$-stability technique and the discrete Gr\"onwall lemma. In one dimension, the theory yields an optimal $L^2$ error estimate of order $h^{k+1}+\tau^p$, i.e., optimal in both space and time, where $p=d$ if $n-d$ is even and $p=d+1$ otherwise; this optimality relies crucially on the projection $\Pi_h$ and on the penalty term on $u_x$.
{ In two dimensions, the analysis predicts only $O(h^{k-1})$ for the spatial error. In two dimensions, the analysis predicts only $O(h^{k-1})$ for the spatial error, while the temporal error remains optimal. The reduction in the spatial order arises because, on triangular meshes, no projection with the favorable properties of the one-dimensional $\Pi_h$ is available to achieve the optimal order, and this loss is reflected in the boundary contributions $J_2(\cdot,\cdot)$ and $B_2(\eta_v,\xi_u)-B_2(\xi_v,\eta_u)$ on $\mathcal{E}_h^D$.} Despite this, numerical experiments confirm optimal convergence of order $h^{k+1}+\tau^p$ in both one- and two-dimensional settings.

Numerical tests verify the theoretical findings and assess the practical performance of the method. The penalty parameter is found to be essential for optimal accuracy; without it, the method fails to achieve optimal convergence. The LBRFD method attains temporal orders $1$ through $5$ for suitable $(n,d)$ pairs. The startup procedures leave both the spatial and temporal convergence orders unchanged. For higher-order cases such as $(n,d)=(8,5), (10,5), (12,5)$, where the scheme is formally sixth-order for ODEs, the lack of a suitable $\eta\in[0,1)$ for $G$-stability prevents a rigorous proof, though numerical results indicate orders close to~$6$. { Comparisons with BDF$p$ methods show that the two approaches attain the same temporal convergence orders; in one dimension, for $p\le 4$, suitable $(n,d)$ choices can yield smaller errors than the corresponding BDF$p$ methods, although this advantage is not observed for $p\ge 5$ or in the two-dimensional case. The main advantages of the LBRFD methods lie in their flexibility and, for some parameter pairs, wider stability angles, at the cost of larger errors, higher memory requirements, and additional startup cost.}

Future work includes extending the method to nonlinear fourth-order equations, refining the spatial error analysis to close the theory--numerics gap, developing adaptive time-stepping strategies, and applying the scheme to other higher-order PDEs such as the Cahn--Hilliard equation.


\section*{Statements and Declarations}

\bigskip\noindent{\bf CRediT authorship contribution statement}

{\bf Hongying Huang:} Writing – review \& editing, Writing - original draft, Validation, Software, Formal analysis, Conceptualization. {\bf Hong Zhang:} Writing – review \& editing, Formal analysis, Software, Validation. {\bf Yanming Zhang:} Conceptualization, Investigation, Writing – review \& editing, Validation, Software.

\bigskip\noindent{\bf Funding}

The first author was supported under  the National Natural Science Foundation of China (grant number 11771398) and Innovation Team Project of Regular Universities in Guangdong Province (grant number 2025KCXTD037). The third author was supported by the Natural Science Foundation of Hunan Province, China (grant number 2025JJ60044).

\bigskip\noindent{\bf Declaration of competing interest}

The authors declare that they have no known competing financial interests or personal relationships that could have appeared to influence the work reported in this paper.

\bigskip\noindent{\bf Data availability}

No datasets were generated or analyzed during the current study. All results presented in this paper are based on numerical simulations implemented by the authors.

\bibliographystyle{elsarticle-num}
\bibliography{high_order_parabolic_DG_review}

%
%
%
%

\end{document}